\documentclass[11pt,leqno]{amsart}
\usepackage{amsmath,amsfonts,amssymb,amscd,amsthm,amsbsy,upref}
\numberwithin{equation}{section}
\def\hangbox to #1 #2{\vskip3pt\hangindent #1\noindent \hbox to #1{#2}$\!\!$}

\newtheorem{thm}{Theorem}[section]
\newtheorem*{thmA}{Theorem A}
\newtheorem*{thmB}{Theorem B}
\newtheorem*{thmC}{Theorem C}
\newtheorem{lem}[thm]{Lemma}
\newtheorem{cor}[thm]{Corollary}
\newtheorem{prop}[thm]{Proposition}
\theoremstyle{definition}

\newtheorem{defn}[thm]{Definition}
\newtheorem{defin}[thm]{Definition}

\newtheorem{tsirelson}[thm]{Tsirelson spaces}

\theoremstyle{remark}
\newtheorem{rem}[thm]{Remark}
\newtheorem{rems}[thm]{Remarks}

\def\N{{\mathbb N}}
\def\R{{\mathbb R}}

\def\bE{{\mathbf E}}
\def\bF{{\mathbf F}}

\def\a{{\rm a}}
\def\b{{\rm b}}

\def\cA{{\mathcal A}}
\def\cB{{\mathcal B}}
\def\cC{{\mathcal C}}
\def\cG{{\mathcal G}}
\def\cL{{\mathcal L}}

\newcommand{\ra}{\rangle}
\newcommand{\la}{\langle}

\def\sfrac#1#2{\kern.1em\raise.5ex\hbox{$#1$}
        \kern-.1em/\kern-.05em\lower.25ex\hbox{$#2$}}

\def\ov#1{\overline{#1}}
\def\To{\Rightarrow}
\def\ds{\displaystyle}
\def\vp{\varepsilon}
\def\dim{\operatorname{dim}}

\def\Tac{{T_{\cA,c}}}

\newcommand\kin{\!\in\!}

\newcommand{\sign}{\text{\rm sign}}

\newcommand{\w}{\text{w}}
\newcommand{\fw}{\text{\fw}}

\newcommand{\rk}{\hbox{\text{\rm rk}}}

\newcommand{\supp}{{\rm supp}}
\newcommand{\coo}{c_{00}}

\newcommand{\cuts}{{\rm cuts}}
\newcommand{\dist}{{\rm dist}}
\newcommand{\f}{{\rm f}}
\newcommand{\spa}{{\rm span}}

\def\ran{\operatorname{ran}}

\newcommand{\Deltab}{\overline{\Delta}}
\newcommand{\Gammab}{\overline{\Gamma}}
\newcommand{\gammab}{{\ov{\gamma}}}
\newcommand{\etab}{{\ov{\eta}}}
\newcommand{\xib}{{\ov{\xi}}}
\newcommand{\Rb}{\ov{R}}
\newcommand{\Jb}{\ov{J}}
\newcommand{\Mb}{\ov{M}}
\newcommand{\Fb}{\ov{F}}
\newcommand{\bFb}{\ov{\bF}}

\begin{document}
\allowdisplaybreaks
\title{The universality of $\ell_1$ as a dual space}
\author{D. Freeman, E. Odell\and  Th. Schlumprecht}
\address{Department of Mathematics, Texas A\&M University, College Station,
TX 77843-3368}
\email{freeman@math.tamu.edu}
\address{Department of Mathematics\\ The University of Texas at Austin,\\ Austin, TX 78712-0257}
\email{odell@math.utexas.edu}
\address{Department of Mathematics, Texas A\&M University, College Station,
TX 77843-3368}
\email{schlump@math.tamu.edu}
\thanks{Research of the last two authors was supported by the
National Science Foundation.}


\begin{abstract}
Let $X$ be a Banach space with a separable dual.
We prove that $X$ embeds isomorphically into a $\cL_\infty$ space $Z$
whose dual is isomorphic to $\ell_1$.
If, moreover, $U$ is a space so that $U$ and  $X$  are totally incomparable, then we construct such a $Z$, 
 so that $Z$ and $U$ are totally incomparable.
If $X$ is separable and reflexive, we show that $Z$ can be made to be
somewhat reflexive.
\end{abstract}
\maketitle

\section{Introduction}\label{S:1}

In 1980 J.~Bourgain and F.~Delbaen \cite{BD} showed the surprising
diversity of $\cL_\infty$ Banach spaces whose duals are isomorphic to $\ell_1$
by constructing such a space $Z$ not containing an isomorph of $c_0$.
Moreover, $Z$ is {\em somewhat reflexive\/}, i.e., every infinite
dimensional subspace of $Z$ contains an infinite dimensional reflexive
subspace.
In fact, R.~Haydon \cite{H} proved the reflexive subspaces could be
chosen to be isomorphic to $\ell_p$ spaces.

The structure of Banach spaces $X$ whose dual is isometric to
$\ell_1$ is more limited. Such a space $X$ must contain $c_0$
\cite{Z1} and in fact be an isometric quotient of $C(\Delta)$
\cite{JZ1}. Finally it was shown in  \cite{F}  that such spaces must
be $c_0$ saturated. Nevertheless, such a space need not be an
isometric quotient of some $C(\alpha)$, for $\alpha<\omega_1$
\cite{A1}.

The construction developed by Bourgain and Delbaen is quite general
and allows for additional modifications. Very recently S.~Argyros
and R.~Haydon \cite{AH} were able to adapt this construction to
solve the famous {\em Scalar plus Compact Problem} by building an
infinite dimensional Banach space, with dual isomorphic to $\ell_1$,
on which all operators
 are a compact perturbation of  a multiple of the identity.
In this paper we will prove three main theorems concerning
isomorphic preduals of $\ell_1$.

\begin{thmA}
Let $X$ be a Banach space with separable dual. Then $X$ embeds into
a $\cL_\infty$ space $Y$ with $Y^*$ isomorphic to $\ell_1$.
\end{thmA}
Moreover, we have the following refinements of Theorem A.
\begin{thmB}
Let $X$ and $U$   be totally incomparable infinite dimensional
Banach spaces with separable duals. Then $X$ embeds into a
$\cL_\infty$ space $Z$ whose dual is isomorphic to $\ell_1$, so that
$Z$ and $U$ are totally incomparable.
\end{thmB}

\begin{thmC}
Let $X$ be a separable reflexive Banach space. Then $X$ embeds into
a somewhat reflexive $\cL_\infty$ space $Z$, whose dual is
isomorphic to $\ell_1$.  Furthermore, if $U$ is a Banach space with
separable dual such that $X$ and $U$ are totally incomparable, then
$Z$ can be chosen to be totally incomparable with $U$.
\end{thmC}

We recall that $X$ and $U$ are called totally incomparable if no
infinite dimensional Banach space embeds into both $X$ and $U$.

Since there are reflexive spaces of arbitrarily high countable
Szlenk index \cite{Sz} Theorem~B (with $U=c_0$) as well as Theorem C
solve a question of Alspach \cite[Question 5.1]{A2} who asked
whether or not there are $\cL_\infty$ spaces with arbitrarily high
Szlenk index not containing $c_0$.  Moreover Alspach, in conference
talks, asked whether Theorem~A could be true.  Furthermore, Theorem
$B$ with $U=c_0$ solves the longstanding open problem of showing
that if $X^*$ is separable and $X$ does not contain an isomorph of
$c_0$, then $X$ embeds into a Banach space with a shrinking basis
which does not contain an isomorph of $c_0$.

In Section \ref{S:2} we review the skeletal aspects of the
Bourgain-Delbaen construction of $\cL_\infty$ spaces, following more
or less, \cite{AH}. Theorem~A  will be proved in Section \ref{S:4},
while  the proofs of Theorems~B and C are presented in Section
\ref{S:5}. The construction used to prove Theorem~A will also be
useful in the case where $X^*$ is not separable. The construction
proving Theorems B and C will be an {\em augmentation} of that used
to prove Theorem~A.

Section \ref{S:3} contains background material necessary for our
proof. We review some embedding theorems from \cite{OSZ2} and
\cite{FOSZ} that play a role in the subsequent constructions.
Terminology and definitions are given along with some propositions
that facilitate their use.  In particular, we define the notion of a
$c$-decomposition and relate it to an FDD being shrinking
(Proposition \ref{P:bimonotone}).  This will be used to show that
our $\cL_\infty$ constructs have dual isomorphic to $\ell_1$.  We
also show how Theorem \ref{P:bimonotone} leads to an alternate
and self contained proof of a less precise version of embedding
Theorems \ref{thm:FOSZ} and \ref{thm:OSZ2}, which is sufficient for
their use in this paper.

We use standard Banach space terminology as may be found in
\cite{JL} or \cite{LT}. We recall that $X$ is $\cL_\infty$ if there
exist $\lambda \!<\!\infty$ and finite dimensional subspaces $E_1
\subseteq E_2 \subseteq \cdots$ of $X$ so that $X =
\overline{\bigcup_{n=1}^\infty} E_n$ and the Banach-Mazur distance
satisfies
$$d(E_n, \ell_\infty^{\dim (E_n)}) \!\le\! \lambda\ ,\ \ \text{ for all }\ \
n\!\in\! \N\ .$$ In this case we say $X$ is $\cL_{\infty,\lambda}$.
 $S_X$ and $B_X$ denote the unit sphere and unit ball of $X$,
respectively. A sequence of finite dimensional subspaces of $X$,
$(E_i)_{i=1}^\infty$ is an FDD (finite dimensional decomposition) if
every $x\in X$ can be uniquely expressed as $x=\sum_{i=1}^\infty
x_i$ where $x_i\in F_i$ for all $i\in\N$.  It is usually required
that $E_i\neq\{0\}$ for all $i\in\N$ for $(E_i)_{i=1}^\infty$ to be
a finite dimensional decomposition, but it will be convenient for us
to allow $E_i=\{0\}$ for some $i$'s in Section \ref{S:5}

We note that there are deep constructions of $\mathcal{L}_\infty$
spaces other then the ones in \cite{BD}. For example Bourgain and
Pisier \cite{BP} prove that every separable Banach space $X$ embeds
into a $\mathcal L_\infty$ space $Y$ so that $Y/X$ is a Schur space
with the Radon Nikodym Property. P.~Dodos \cite {D} used the
Bourgain-Pisier construction to prove that for every $\lambda>1$
there exists a class $(Y^\xi_\lambda)_{\xi<\omega_1}$ of separable
$\cL_{\infty,\lambda}$ spaces with the following properties. Each
$Y^\xi_\lambda$ is non-universal (i.e. $C[0,1]$ does not embed into
$Y^\xi_\lambda$) and
 if $X$ is separable with $\phi_{NU}(X)\le \xi$, then $X$ embeds into $Y^\lambda_\xi$. Here $\phi_{NU}$ is
 Bourgain's ordinal index based on the Schauder basis for $C[0,1]$.
 Now $C[0,1]$ is a $\cL_\infty$-space and is universal for the class of separable Banach spaces. Theorem~A yields that the class of $\cL_\infty$-spaces with separable dual is universal for the class of all
 Banach spaces with separable dual.

We thank the referee for providing very useful suggestions, which
simplified and expanded some results in our original version.

\section{Framework of the Bourgain-Delbaen construction}\label{S:2}

As promised, this section contains the general framework of the construction of {\em Bourgain-Delbaen spaces}.
This framework is general enough to include the original space of Bourgain and  Delbaen \cite{BD},
 the spaces constructed in \cite{AH},  as well as the spaces constructed in this paper.
We follow, with slight changes and some notational differences, the presentation in
 \cite{AH} and start by introducing  {\em Bourgain-Delbaen sets}.

 \begin{defin}\label{D:1.1} (Bourgain-Delbaen-sets)
  A sequence of finite sets $(\Delta_n:n\kin\N)$ is called a {\em Sequence of Bourgain-Delbaen Sets} if it satisfies the following
  recursive conditions:

 $\Delta_1$ is any finite set, and assuming that for some $n\kin\N$ the sets $\Delta_1$, $\Delta_2$,$\ldots$, $\Delta_n$ have been chosen,
 we let $\Gamma_n=\bigcup_{j=1}^n \Delta_j$. We denote the unit vector basis of $\ell_1(\Gamma_n)$  by $(e^*_\gamma:\gamma\kin\Gamma_n)$,
 and consider the spaces $\ell_1(\Gamma_j)$ and $\ell_1(\Gamma_n\setminus \Gamma_j)$, $j<n$,   to be, in the natural way, embedded into
 $\ell_1(\Gamma_n)$.

 For $n\ge 1$, $\Delta_{n+1}$ will be the union of two sets $\Delta_{n+1}^{(0)}$ and
 $\Delta_{n+1}^{(1)}$, where $\Delta_{n+1}^{(0)}$ and  $\Delta_{n+1}^{(1)}$ satisfy the following
 conditions.

 The set $\Delta_{n+1}^{(0)}$ is finite and
 \begin{align}\label{E:1.1.1}
 \Delta_{n+1}^{(0)}\subset \big\{(n+1,\beta,b^*,f): \beta\kin[0,1], b^*\kin B_{\ell_1(\Gamma_n)}, \text{ and }f\kin V_{(n+1,\beta,b^*)}\big\},
 \end{align}
 where $V_{(n+1,\beta,b^*)}$ is a finite set for $\beta\kin[0,1]$ and $b^*\kin B_{\ell_1(\Gamma_n)}$.

$\Delta^{(1)}_{n+1}$ is finite and
 \begin{equation}\label{E:1.1.2}
\Delta_{n+1}^{(1)} \subset \left\{ (n+1, \alpha, k,\xi, \beta, b^*,f):
 \begin{matrix} \alpha,\beta\kin[0,1], k\kin\{1,2,\ldots n-1\}, \xi\kin \Delta_k, \\
b^*\kin B_{\ell_1(\Gamma_n\setminus \Gamma_k)} \text{ and } f\kin
V_{(n+1,\alpha, k,\xi,\beta,b^*) }\end{matrix}
\right\},\end{equation} where $V_{(n+1,\alpha, k,\xi,\beta,b^*)}$ is
a finite set for $\alpha\kin[0,1]$, $k\kin\{1,2,\ldots, n-1\}$,
$\xi\kin\Delta_k$,  $\beta \kin[0,1]$, and  $b^*\kin
B_{\ell_1(\Gamma_n\setminus\Gamma_k)}$.

Moreover, we assume that  $\Delta_{n+1}^{(0)}$ and
$\Delta_{n+1}^{(1)}$ cannot both be empty.

If $(\Delta_n)$ is a sequence of Bourgain-Delbaen sets  we put $\Gamma=\bigcup_{j=1}^\infty \Gamma_n$.
For $n\kin\N$,  and $\gamma\kin\Delta_n$ we call $n$ the {\em rank of $\gamma$} and denote it by $\rk(\gamma)$. If $n\ge 2$
and $\gamma=(n,\beta,b^*,f)\kin \Delta^{(0)}_n$,
 we say that $\gamma$ is {\em of type $0$,} and if   $\gamma=(n, \alpha, k,\xi, \beta, b^*,f)\kin  \Delta^{(1)}_n$, we
  say that  $\gamma$ is {\em of type $1$}. In both cases we  call
 $\beta$ {\em the weight of $\gamma$} and denote it by $\w(\gamma)$ and
   call $f$  the {\em free variable} and denote it by $\f(\gamma)$.

   In case that  $V_{(n+1,\beta,b^*)}$ or $V_{(n+1,\alpha, k,\xi,\beta,b^*)}$ is a singleton
   (which will be often he case)
   we sometimes suppress  the dependency in the free variable and write
   $(n+1,\beta,b^*)$ instead of $(n+1,\beta,b^*,f)$ and $(n+1,\alpha, k,\xi,\beta,b^*)$
   instead of $(n+1,\alpha, k,\xi,\beta,b^*,f)$.
 \end{defin}
Referring to a sequence of sets $(\Delta_n:n\kin\N)$ as
Bourgain-Delbaen sets we will always  mean that  the  sets
$\Delta_n^{(0)}$, $\Delta_n^{(1)}$, $\Gamma_n$ and $\Gamma$ have
been  defined satisfying the conditions above. We consider the
spaces $\ell_\infty(\bigcup_{j\kin A} \Delta_j)$ and
$\ell_1\big(\bigcup_{j\kin A} \Delta_j\big)$,  for $A\subset \N$, to
be naturally embedded into $\ell_\infty(\Gamma)$
 and $\ell_1(\Gamma)$, respectively.

We denote by $c_{00}(\Gamma)$ the real vector space of families $x=(x(\gamma):\gamma\kin\Gamma)\subset\R$ for which
 the {\em support},
 $\supp(x)=\{\gamma\kin\Gamma: x(\gamma)\not=0\}$, is finite.
  The  unit vector  basis of  $c_{00}(\Gamma)$ is denoted by $(e_\gamma:\gamma\kin\Gamma)$, or, if we regard $c_{00}(\Gamma)$ to be a subspace
   of a dual space, such as $\ell_1(\Gamma)$, by $(e^*_\gamma:\gamma\kin\Gamma)$.
  If $\Gamma=\N$ we write $c_{00}$ instead of $c_{00}(\N)$.

\begin{defin}\label{D:1.3} (Bourgain-Delbaen families of functionals)

Assume that $(\Delta_n:n\kin\N)$ is a sequence  of Bourgain-Delbaen sets.
By induction on $n$ we will define for all $\gamma\kin\Delta_n$, elements
 $c^*_\gamma\kin \ell_1(\Gamma_{n-1})$ and
 $d^*_\gamma\kin\ell_1(\Gamma_n)$, with $d^*_\gamma=e^*_\gamma-c^*_\gamma$.

 For $\gamma\kin\Delta_1$ we define $c^*_\gamma=0$, and thus $d^*_\gamma=e^*_\gamma$.

 Assume that for some $n\kin\N$ we have defined $(c^*_\gamma:\gamma\kin\Gamma_n)$,
 with $c^*_\gamma\kin\ell_1(\Gamma_{j-1})$, if $j\le n$ and $\rk(\gamma)=j$. It follows therefore
 that $(d_\gamma^*:\gamma\kin\Gamma_n)=(e^*_\gamma-c^*_\gamma:\gamma\kin\Gamma_n)$ is a basis
 for $\ell_1(\Gamma_n)$ and thus for $k\le n$ we have projections:
 \begin{equation}\label{E:1.3.1}
 P^*_{(k,n]}:\ell_1(\Gamma_n)\to \ell_1(\Gamma_n) ,\quad \sum_{\gamma\in\Gamma_n} a_\gamma d^*_\gamma  \to \sum_{\gamma\in \Gamma_n\setminus \Gamma_k}a_\gamma d^*_\gamma.
 \end{equation}
For $\gamma\kin\Delta_{n+1}$ we define
\begin{equation}\label{E:1.3.2}
c_\gamma^*=\begin{cases} \beta b^*
       &\text{if $\gamma=(n+1,\beta,b^*,f)\kin \Delta_{n+1}^{(0)}$,}\\
                                                  \alpha e^*_\xi +\beta P^*_{(k,n]}(b^*) &\text{if $\gamma=(n+1,\alpha,k,\xi,\beta,b^*,f)\kin \Delta_{n+1}^{(1)}$.}
\end{cases}
\end{equation}
We call $(c^*_\gamma:\gamma\kin\Gamma)$,  the {\em Bourgain-Delbaen family of functionals associated to $(\Delta_n: n\kin\N)$}. We will, in this case,
 consider the projections $P^*_{(k,n]}$ to be defined on all of $c_{00}(\Gamma)$, which is possible since $(d^*_\gamma:\gamma\kin\Gamma)$
  forms a  vector basis of $c_{00}(\Gamma)$  and, (as we will observe  later)  under further assumptions, a Schauder basis of $\ell_1(\Gamma)$.
 \end{defin}
 \begin{rems}   The reason for using $*$ in the notation for $P^*_{(k,m]}$ is that later
  we will show (with additional assumptions) that the $P^*_{(k,m]}$'s are the adjoints  of coordinate projections $P_{(k,m]}$ on a space
  $Y$ with an FDD $\bF=(F_j)$ onto $\oplus_{j\in(k,m]} F_j$.

 Of course we could,  in the definition of $\Delta^{(0)}_{n+1}$ and $\Delta^{(1)}_{n+1}$, assume $\beta=1$, rescale $b^*$ accordingly, possibly increasing   the number
 of free variables,  then simply define $c^*_\gamma=b^*$, if $\gamma$ is of type $0$,
 or $c^*_\gamma=\alpha e^*_\xi+P^*_{(k,n]}(b^*)$, if $\gamma$ is of type $1$.
 Nevertheless, it  will prove later more convenient to have this  redundant representation which
  will allow us to change the weights of the elements of $\Gamma$ and rescale the $b^*$'s, without changing the $c^*_\gamma$'s. Moreover, it will be useful   for recognizing that our framework is a generalization of the  constructions in \cite{AH} and \cite{BD}.
 \end{rems}
 The next observation is a slight generalization of a result in \cite{AH}, the main idea tracing back to  \cite{BD}.

\begin{prop}\label{P:1.4}
Let $(\Delta_n:n\kin\N)$ be a sequence of  Bourgain-Delbaen sets and
let $(c^*_\gamma:\gamma\kin \Gamma)$
 be the  corresponding   family of associated  functionals. Let  $(P^*_{(k,m]}: k<m)$   and $(d^*_\gamma:\gamma\kin\Gamma)$ be  defined as  in Definition \ref{D:1.3}.
Thus
 $$   P^*_{(k,n]}: c_{00}(\Gamma)\to c_{00}(\Gamma) ,\quad \sum_{\gamma\in\Gamma} a_\gamma d^*_\gamma  \to \sum_{\gamma\in \Gamma_n\setminus \Gamma_k}a_\gamma d^*_\gamma.$$
 For $n\kin\N$, let
 $F^*_n\!=\!\spa(d^*_\gamma:\gamma\kin\Delta_n)$
  and  for $\theta\kin[0,1/2)$ let $C_1(\theta)=C_1=0$ and if $n\ge 2$,
 $$C_n(\theta) =\sup \big\{ \beta\|P^*_{(k,m]}(b^*)\|: \gamma=(\tilde n,\alpha,k,\xi,\beta, b^*,f)\kin \Delta^{(1)}_{\tilde n}, k<m<\tilde n\le n, \beta>\theta\big\},$$
 with $\sup(\emptyset)=0$, and
 $$C_n=C_n(0)=\sup \big\{ \beta\|P^*_{(k,m]}(b^*)\|: \gamma= (\tilde n,\alpha,k,\xi,\beta, b^*,f)\kin \Delta^{(1)}_{\tilde n}, k<m<\tilde n\le n \big\}.$$
 Then
   \begin{equation}\label{E:1.4.1}
   \oplus_{j=1}^n F^*_j=\spa(e^*_\gamma:\gamma\kin\Gamma_n)=\ell_1(\Gamma_n),
 \end{equation}
 and if $C=\sup_n C_n< \infty$,
 then  $\bF^*=(F^*_n)$ is an FDD for $\ell_1(\Gamma)$  whose decomposition constant  $M$  is not larger than $1+C$.  Moreover, for $n\kin\N$ and $\theta\!<\!1/2$,   \begin{equation}\label{E:1.4.1a}
  C_n\le \max\big(2\theta /(1-2\theta), C_n(\theta)\big).
  \end{equation}
 \end{prop}

 \begin{proof} As already noted,  since $d^*_\gamma=e^*_\gamma-c^*_\gamma$, and $c^*_\gamma\kin\ell_1(\Gamma_{n-1})$, for $n\kin\N$ and  $\gamma\kin\Delta_n$, \eqref{E:1.4.1}
 holds.
 By induction on $n\kin\N$ we will show that  for all $0\le  m< n$, $\|P_{[1,m]} ^*|_{\ell_1(\Gamma_n)} \|  \le  1+ C_{n}$, and that
  \eqref{E:1.4.1a} holds, whenever $\theta<1/2$.
   For $n=1$, and thus $m=0$ and $C_1=0$, the claim follows trivially ($\|P^*_\emptyset\|\equiv 0$).
 Assume the claim is true for some $n\kin\N$.  Using the induction hypothesis and the
   fact that every element of $B_{\ell_1(\Gamma_{n+1})}$ is a   convex combination  of
    $\{\pm e^*_\gamma: \gamma\kin\Gamma_{n+1}\}$ and $C_n(\theta)\le C_{n+1}(\theta)$,
  it is enough to show that  for all $\gamma\kin\Delta_{n+1}$ and all $m\le n$
  \begin{align}\label{E:1.4.2}
  &\|P^*_{[1,m]}(e^*_\gamma)\|\le 1+C_{n+1}\text{ and}\\
  \label{E:1.4.3}&\|\beta P^*_{(k,m]}(b^*)\|\le \frac{2\theta}{1\!-\!2\theta}\vee C_{n}(\theta),\text{if $\beta\!\le\! \theta\!<\!1/2$ and $\gamma=(n\!+\!1,\alpha,k,\xi,\beta,b^*,f)\kin \Delta^{(1)}_{n+1}$}.
  \end{align}

  According
   to  \eqref{E:1.3.2}  we can write
   $$e^*_\gamma=d^*_\gamma+c^*_\gamma=d^*_\gamma+ \alpha e^*_\xi +\beta P^*_{(k,n]}(b^*),$$
   with $\alpha,\beta\kin[0,1]$, $0\!\le\!k\!<\!n$, $\xi\kin \Delta_k$ (put $k=0$ and $\alpha=0$ if $\gamma$ is of type $0$),      and $b^*\kin B_{\ell_1(\Gamma_n\setminus \Gamma_k)}$.

  Thus
  $$  P^*_{[1,m]}(e^*_\gamma)=\alpha  P^*_{[1,m]}(e^*_\xi) +\beta P^*_{(\min(m,k),m]}(b^*).$$

Now, if $k\ge m$, then   $P^*_{[1,m]}(e^*_\gamma)=\alpha P^*_{[1,m]}(e^*_\xi)$ and thus
  our claim \eqref{E:1.4.2} follows from the induction hypothesis:
  $$\|\alpha P^*_{[1,m]}(e_\xi^*)\|\le 1+C_k\le 1+C_{n+1}.$$

  If $k<m$ it follows, again using the induction hypothesis in the type 0 case, that
$$  \|P^*_{[1,m]}(e^*_\gamma)\|\le \alpha \|e^*_\xi\| +\beta \|P^*_{(k,m]}(b^*)\| \le 1 + C_{n+1},
\text{ which yields \eqref{E:1.4.2}}.$$

In order to show \eqref{E:1.4.3}, let  $\gamma=(n+1,\alpha,k,\xi,\beta,b^*,f)\kin \Delta^{(1)}_{n+1}$, with
$\beta\le \theta<1/2$. We deduce from the induction hypothesis  that
\begin{align*}
\|\beta P^*_{(k,m]}(b^*)\|&\le \beta( \|P^*_{[1,m]}|_{\ell_1(\Gamma_n)} \|+ \|P^*_{[1,k]}|_{\ell_1(\Gamma_n)} \|)\\
                                            &\le 2 \theta( C_n+1)\\
                                            &\le \begin{cases}    2\theta\big( C_n(\theta)+1))
                                            \le  2\theta C_n(\theta)+ C_n(\theta)(1-2\theta) =C_n(\theta)
                                             &\text{if $C_n(\theta)>     \frac{2\theta}{1-2\theta} $,}\\
                                         2\theta\Big( \frac{2\theta}{1-2\theta} +1\Big)   =  \frac{2\theta}{1-2\theta}   &\text{otherwise,} \end{cases}\\
                                                       & \le \max\Big(\frac{2\theta}{1-2\theta}, C_n(\theta)\Big).
\end{align*}
This  finishes the induction step, and  hence the proof. \end{proof}

 \begin{rems}\label{R:1.4a} Let  $\Gamma$ be linearly  ordered as $(\gamma_j:j\kin\N) $ in such a way that $\rk(\gamma_i)\le \rk(\gamma_j)$, if $i\le j$. Then the same arguments
  show that, under the assumption $C\!<\!\infty$ stated in Proposition \ref{P:1.4}, $(d^*_{\gamma_j})$ is  actually a Schauder basis of $\ell_1$ \cite{AH}.  But, for our purpose,   the  FDD is the more useful coordinate system.

The spaces constructed in  \cite{AH}  satisfy the condition that for some $\theta<1/2$  we have  $\beta \le \theta$, for all  $\gamma=(n,\alpha,k,a^*,\beta,b^*,f)\kin \Gamma$ of type 1.
 Thus in that case  $C_n(\theta)=0$, $n\kin\N$, and  the conclusion of Proposition \ref{P:1.4} is  true for $C\le 2\theta/(1-2\theta)$ and, thus   $M\le 1/(1-2\theta)$.

 The Bourgain-Delbaen sets  we will consider  in later sections will satisfy the following condition for some $0<\theta<1/2$:
 \begin{align}\label{E:1.1a}
  &\text{For  each $n\in\N$ and $\gamma=(n,\alpha,k,\xi,\beta, b^*,f)\in \Delta_n^{(1)}$,}\\
 &\text{either $\beta\le \theta$, or   $b^*=e^*_\eta$ for some $\eta\in\Delta_m$, $k<m<n$, such
  that $c^*_\eta=0$.}\notag\end{align}
  Note that in the second case it follows that
   $e^*_\eta=d^*_\eta$ and so $P^*_{(k,m]}(e^*_\eta)=e^*_\eta$.
 Thus,
   $\beta\|P^*_{(k,m]}(b^*)\|=\beta\|e^*_\eta\|\le 1$, and thus, we deduce that the assumptions of Proposition \ref{P:1.4}  are satisfied, namely  that
  $\bF^*$ is an FDD of $\ell_1$  whose decomposition constant $M$  is not larger than $\max (1/(1-2\theta),2)$.

\end{rems}

 Assume we are given  a sequence of Bourgain-Delbaen sets  $(\Delta_n:n\kin\N)$, which satisfy the assumptions
  of Proposition  \ref{P:1.4} with $C<\infty$ and let $M$ be the decomposition constant of  the FDD $(F_n^*)$
 in $\ell_1(\Gamma)$.
 We now  define the {\em Bourgain-Delbaen space associated to } $(\Delta_n:n\kin\N)$.
   For a  finite or cofinite  set $A\!\subset\!\N$, we let $P^*_A$  be the projection of $\ell_1(\Gamma)$ onto  the subspace $\oplus_{j\in A} F_j^*$  given by
  $$P^*_A:\ell_1(\Gamma)\to \ell_1(\Gamma),\qquad \sum_{\gamma\in\Gamma} a_\gamma d^*_\gamma\mapsto
   \sum_{\gamma\in A} a_\gamma d^*_\gamma.$$
   If $A=\{m\}$, for some  $m\kin\N$, we write $P_m^*$ instead of $P_{\{m\}}^*$.  For $m\kin\N$, we denote
   by $R_m$ the restriction operator  from $\ell_1(\Gamma)$ onto  $\ell_1(\Gamma_m)$ (in terms of the basis $(e^*_\gamma)$)   as well
     the usual restriction operator from $\ell_\infty(\Gamma)$ onto  $\ell_\infty(\Gamma_m)$. Since $R_m\circ P^*_{[1,m]}$ is a  projection
      from $\ell_1(\Gamma)$ onto $\ell_1(\Gamma_m)$, for $m\kin\N$, it follows that
      the map
      $$J_m:\ell_\infty(\Gamma_m)\to \ell_\infty(\Gamma), \quad x\mapsto P^{**}_{[1,m]}\circ R_m^*(x),$$
      is an isomorphic embedding ($ P^{**}_{[1,m]}$ is the adjoint of $P^*_{[1,m]}$ and, thus, defined on $\ell_\infty(\Gamma)$).
       Since $R^*_m$ is the natural embedding of $\ell_\infty(\Gamma_m)$ into   $\ell_\infty(\Gamma)$
      it follows, for all $m\kin\N$, that
      \begin{align}
      \label{E:1.2a}  &R_m\circ J_m(x)=x, \text{ for $x\in\ell_\infty(\Gamma_m)$, thus $J_m$ is an extension operator,}\\
          \label{E:1.2c} &J_n\circ R_n \circ J_m(x)= J_m(x), \text{ whenever $m\le n$ and $x\in\ell_\infty(\Gamma_m$)},
          \intertext{and by Proposition \ref{P:1.4},}
           \label{E:1.2b}  &\|J_m\| \le  M.
      \end{align}
    Hence the spaces $Y_m=J_m(\ell_\infty(\Gamma_m))$, $m\kin\N$, are  finite-dimensional nested subspaces of $\ell_\infty(\Gamma)$ which (via $J_m)$ are $M$-isomorphic images
    of   $\ell_\infty(\Gamma_m)$. Therefore
  $Y=\overline{\bigcup_{m\in\N} Y_n}^{\ell_\infty}$
    is a ${\cL}_{\infty,M}$ space. We call $Y$ the {\em Bourgain-Delbaen space associated to $(\Delta_n)$}. It follows from the definition
    of $Y$, and from \ref{E:1.2a}, that for any $x\in\ell_\infty(\Gamma)$ we have
\begin{equation}\label{E:1.3}
 x\in Y \iff  x= \lim_{m\to\infty}\|x- J_m\circ R_m (x)\|=0.
\end{equation}
    Define for $m\kin\N$
    $$P_{[1,m]}:Y\to Y,\quad x\mapsto J_m\circ R_m(x).$$
    We claim that $P_{[1,m]}$ coincides  with the   restriction of the  adjoint $P^{**}_{[1,m]}$ of $P^{*}_{[1,m]}$
   to the space $Y.$
    Indeed, if $n\kin\N$, with $n\ge m$, and $x=J_n(\tilde x)\kin Y_n$, and $b^*\in\ell_1(\Gamma)$ we have that
      \begin{align*}
      \la P^{**}_{[1,m]}(x), b^*\ra&=  \la x, P^*_{[1,m]}(b^*)\ra\\
      &=\la R_m  (x) ,R_m\circ P^*_{[1,m]} (b^*)\ra \text{ (since  $P^*_{[1,m]} (b^*)\in\spa(e^*_\gamma:\gamma\in \Gamma_m)$)}\\
       &=\la   P^{**}_{[1,m]}\circ R_m^* \circ R_m(x),b^*\ra= \la P_{[1,m]}(x), b^*\ra.
      \end{align*}
    Thus our claim follows since $\bigcup_n Y_n$ is dense in $Y.$

   We therefore deduce that $Y$ has an FDD $(F_m)$, with $F_m=(P_{[1,m]}-P_{[1,m-1]})(Y)$, and
    as we observed in \eqref{E:1.2b},  $Y_m=\oplus_{j=1}^n F_j$ is, via $J_m$, $M$-isomorphic to $\ell_\infty(\Gamma_m)$ for $m\kin\N$.   Moreover, denoting
    by $P_A$ the coordinate projections from $Y$ onto $\oplus_{j\in A} F_j$, for all finite or cofinite sets $A\subset \N$,
 it follows that
   $P_A$  is  the adjoint of $P^*_A$ restricted to $Y$, and  $P_A^*$ is the  adjoint of $P_A$ restricted to the  subspace of $Y^*$ generated by the $F_n^*$'s.

   As the next  observation shows,  $J_m|_{\ell_\infty(\Delta_m)}$ is actually an isometry for $m\in\N$.
    \begin{prop}\label{P:1.4a}
  For every $m\in\N$ the map  $J_m|_{\ell_\infty(\Delta_m)}$ is an isometry between $\ell_\infty(\Delta_m)$ (which we consider naturally embedded into $\ell_\infty(\Gamma_m)$)
   and $F_m$.
  \end{prop}
\begin{proof}
Since $J_m(\ell_\infty(\Delta_m))= (J_m-J_{m-1})(\Delta_m)=F_m$, for
$m\in\N$,    $J_m$ is an isomorphism between $\ell_\infty(\Delta_m)$
and $F_m$. By \ref{E:1.2a}, for $x\in\ell_\infty(\Delta_m)$,
$\|J_m(x)\|\ge \|x\|$. In order to finish the proof we will show by
induction on $n\in\N$ that $|e^*_\gamma(J_m(x))|\le 1$ for all
$\gamma\in\Delta_n$ and $x\in\ell_\infty(\Delta_m)$, $\| x\| \le 1$.

If $n\le m$ this is clear since $R_m\circ J_m(x)=x$.  Let $n>m$ and
assume our claim is true for all $\gamma\in\Gamma_n$. Let
$\gamma\in\Delta_{n+1}$ and write $e^*_\gamma$ as $e^*_\gamma=\alpha
e^*_\xi +\beta P^*_{(k,n]}(b^*) + d_\gamma^*$, with
$\alpha\in[0,1]$, $k<n$, $e^*_\xi\in\Delta_k$, and $b^*\in
B_{\ell_1(\Gamma_n\setminus \Gamma_k)}$  ($\alpha=0$, $k=0$, and
replace $e^*_\xi$ by $0$ if $\gamma$ is of type 0).
  We have  for $x\in \ell_\infty(\Delta_m)$, with $\|x\|\le 1$,
\begin{align*}
\la e^*_\gamma, J_m(x)\ra&= \la P^*_{[1,m]}(e^*_\gamma), R^*_m(x)\ra\\
&= \begin{cases} \beta\la  P^*_{(k,m]}(b^*), R^*_m(x)\ra = \beta \la P_{[1,m]}^{*}(b^*) ,R^*_m(x)\ra  =\beta\la  b^*,J_m (x)\ra  &\text{if $k<m$ } \\
                               \alpha \la e^*_\xi,R_m^*(x)\ra =\alpha \la  P^*_{[1,m]}(e^*_\xi) ,      R_m^*(x)\ra= \alpha\la e^*_\xi ,J_m(x)\ra                                  &\text{if $k\ge m$.} \end{cases}
\end{align*}
Where the first equality in the first case holds since
  $\la P_{[1,k]}^{*}(b^*) ,R^*_m(x)\ra=0$. Using our induction hypothesis, this implies our claim.

\end{proof}

   Denote by $\|\cdot\|_*$ the  dual  norm  of $Y^*$.
      \begin{prop}\label{P:1.4b} For all $y^*\in\ell_1(\Gamma)$
   \begin{equation}\label{E:1.4b.1}
   \|y^*\|_*\le \|y^*\|_{\ell_1}\le M\|y^*\|_*.
   \end{equation}
   and if $y^*\in\oplus_{j=m+1}^n F_j^*$, with $0<m<n$, then there is
   a family $(a_\gamma)_{\gamma\in\Gamma_n\setminus \Gamma_m} $
   so that
   \begin{equation}\label{E:1.4b.2}
   y^*=P^*_{(m,n]}\Big(\sum_{\gamma\in\Gamma_n\setminus\Gamma_m} a_\gamma e^*_\gamma\Big)
   \text {and }
   \Big\|\sum_{\gamma\in\Gamma_n\setminus\Gamma_m} a_\gamma e^*_\gamma\Big\|_{\ell_1}\le M\|y^*\|_*.
   \end{equation}
   \end{prop}

\begin{proof}
   The first inequality  in \eqref{E:1.4b.1} is trivial.
   To show the second inequality we let $y^*\in \ell_1(\Gamma_n)$ for some $n\kin\N$ and choose $x\in S_{\ell_\infty(\Gamma_n)}$ so that
   $\la y^*,x\ra=\|y^*\|_{\ell_1}$. Then, from \eqref{E:1.2b} and
   \eqref{E:1.2a},
   $$\|y^*\|_*\ge \Big\la y^*,\frac1M  J _n(x)\Big\ra= \frac1M  \|y^*\|_{\ell_1}.$$

If $y^*\in\oplus_{j=m+1}^n F_j^*$ , we can write $y^*$ as
$$y^*=\sum_{\gamma\in\Gamma_n} \alpha_\gamma e^*_\gamma.$$
Since $P^*_{(m,n]}(e^*_\gamma)=0$, for $\gamma\in\Gamma_m$,   we obtain
$$y^*=P^*_{(m,n]}(y^*)=P^*_{(m,n]}\Big(\sum_{\gamma\in\Gamma_n\setminus\Gamma_m} a_\gamma e^*_\gamma\Big).$$
Moreover we obtain, from \eqref{E:1.4b.1}, that
$$\Big\|     \sum_{\gamma\in\Gamma_n\setminus\Gamma_m} a_\gamma e^*_\gamma    \Big\|_{\ell_1}\le
\Big\|     \sum_{\gamma\in\Gamma_n} a_\gamma e^*_\gamma    \Big\|_{\ell_1}=\|y^*\|_{\ell_1}\le M\|y^*\|_*,$$
which yields \eqref{E:1.4b.2}.
\end{proof}

 We now recall some more notation  introduced in \cite{AH}.  Assume that we are given a Bourgain-Delbaen sequence $(\Delta_n)$ and associated Bourgain-Delbaen
 family of functionals
$(c_\gamma^*\!:\!\gamma\kin\Gamma)$, corresponding to the
Bourgain-Delbaen space $Y$, which admits a  decomposition constant
$M<\infty$. As above we denote its FDD by $(F_n)$. For $n\kin\N$ and
$\gamma\in\Delta_n$, we have
\begin{align*}
e^*_\gamma=d^*_\gamma+c^*_\gamma= d^*_\gamma+\begin{cases}   \beta  b^*
&\text{if   $\gamma=(n, \beta,b^*,f)\in \Delta^{(0)}_n$,}\\
\alpha e^*_\xi\!+\!\beta P^*_{(k,n]}(b^*) &\text{if  $\gamma=(n,
\alpha, k, \xi, \beta, b^*,f)\in \Delta^{(1)}_n$}.
\end{cases}
\end{align*}
By iterating   we eventually arrive (after finitely many steps)  to  a functional of type 0.
By an easy induction  argument we therefore obtain

\begin{prop}\label{P:1.5} For all $n\kin\N$ and $\gamma\in\Delta_n$, there are
$a \in \N$,  $\beta_1,\beta_2,\ldots \beta_a\kin[0,1]$,
$\alpha_1,\alpha_2,\ldots \alpha_a\kin[0,1]$ and  numbers
$0\!=\!p_0< p_1<p_2-1< p_2<p_3<p_3-1,\ldots <p_{a-1}<p_a-1<p_a=n$ in
$\N_0$, vectors $b^*_j$, $j=1,2\ldots a$, with $b^*_j\in
B_{\ell_1(\Gamma_{p_j-1}\setminus \Gamma_{p_{j-1}})}$,
 and $(\xi_j)_{j=1}^a\subset \Gamma_n$, with $\xi_j\in \Delta_{p_j}$, for $j=1,2\ldots a$, and $\xi_a=\gamma$, so that
\begin{equation}\label{E:1.5.1} e^*_\gamma=
 \sum_{j=1}^{a}  \alpha_j d^*_{\xi_j} +  \beta_j P^*_{(p_{j-1},p_j)}(b^*_j) .
 \end{equation}
 Moreover for $1\le j_0< a$
 \begin{equation}\label{E:1.5.1a} e^*_\gamma= \alpha_{j_0} e^*_{\gammab_{j_0}}+  \sum_{j=j_0+1}^{a}  \alpha_j d^*_{\xi_j} +  \beta_j P^*_{(p_{j-1},p_j)}(b^*_j) .
 \end{equation}
 \end{prop}

 We call the representations in \eqref{E:1.5.1} and \eqref{E:1.5.1a}  {\em the analysis of $\gamma$} and {\em partial analysis of $\gamma$}, respectively   and let
 $\cuts(\gamma)=\{p_1,p_2,\ldots p_a\}$,
   which we call the  {\em set  of cuts of $\gamma$}.


\section{Embedding background and other preliminaries}\label{S:3}

Our constructions will depend heavily on some known embedding theorems.
We review these in this section and add a bit more to facilitate their use.
M.~Zippin \cite{Z2} proved that if $X^*$ is separable, then $X$ embeds into a
space with a shrinking basis.
So, in proving Theorem~A, we could begin with such a space.
However, to  make our construction work, we need a quantified version
of this theorem which appears in \cite{FOSZ}.
For Theorem~C, we need a quantified reflexive version \cite{OSZ2}.
We begin with some notation and terminology.

Let $\bE = (E_i)_{i=1}^\infty$ be an FDD for a Banach space $Z$.
$c_{00} (\oplus_{i=1}^\infty E_i)$ denotes the linear span of the $E_i$'s
and if $B\subseteq \N$, $c_{00}(\oplus_{i\in B} E_i)$ is the linear
span of the $E_i$'s for $i\!\in\! B$.
{\em $P_n = P_n^{\bE} :Z\to E_n$ is the $n^{th}$ coordinate projection
for the FDD\/}, i.e., $P_n(z) = z_n$ if $z = \sum_{i=1}^\infty z_i\!\in\! Z$
with $z_i\!\in\! E_i$ for all $i$.
For a finite set or interval $A\subseteq \N$,
$P_A = P_A^{\bE} \equiv \sum_{n\in A} P_n^{\bE}$.
The {\em projection constant\/} of $(E_n)$ in $Z$ is
\begin{equation*}
K = K(\bE,Z) = \sup \left\{ \|P_{[m,n]}^{\bE}\| : m\!\le\! n\right\}\ .
\end{equation*}
$\bE$ is {\em bimonotone\/} if $K(\bE,Z) =1$.

The vector space $c_{00} (\oplus_{i=1}^\infty E_i^*)$, where $E_i^*$ is
the dual space of $E_i$, is naturally identified as a $\omega^*$-dense
subspace of $Z^*$.
Note that the embedding of $E_i^*$ into $Z^*$ is not, in general,
an isometry unless $K(\bE,Z)=1$.
Now we will often be dealing with a bimonotone FDD (via renorming) but
when not we will consider $E_i^*$ to have the norm it inherits as a
subspace of $Z^*$.
We write $Z^{(*)} = [c_{00} (\oplus_{i=1}^\infty E_i^*)]$.
So $Z^{(*)} = Z^*$ if $(E_i)_{i=1}^\infty$ is shrinking, and then
$\bE^*= (E_i^*)_{i=1}^\infty$ is a boundedly complete FDD for $Z^*$.

For $z\!\in\! c_{00} (\oplus_{i=1}^\infty E_i)$ the {\em support of $z$,
$\supp_{\bE}(z)$}, is given by $\supp_{\bE}(z) = \{n:P_n^{\bE} (z) \ne0\}$, and
the {\em range of $z$}, $\ran_{\bE}(z)$ is the smallest interval $[m,n]$
in $\N$ containing $\supp_{\bE} (z)$.

A sequence $(z_i)_{i=1}^\ell$, where $\ell \!\in\! \N$ or $\ell=\infty$,
in $c_{00} (\oplus_{i=1}^\infty E_i)$ is called a {\em block sequence\/}
of $(E_i)$ if $\max \supp_{\bE} (z_n) \!<\! \min \supp_{\bE} (z_{n+1})$ for
all $n\!<\!\ell$.
We write $z_n \!<\! m$ to denote $\max \supp_{\bE}(z_n) \!<\! m$ and
$z_n > m$ is defined by $\min\supp_{\bE}(z_n) > m$.

\begin{defn}\label{D:OSZ1}\cite{OSZ1}
Let $Z$ be a Banach space with an FDD $\bE = (E_i)_{i=1}^\infty$.
Let $V$ be a Banach space with a normalized 1-unconditional basis
$(v_i)_{i=1}^\infty$, and let $1\!\le\! C \!<\!\infty$.
We say that {\em $(E_n)_{n=1}^\infty$ satisfies subsequential $C$-$V$-upper
estimates\/} if whenever $(z_i)_{i=1}^\infty$ is a normalized block
sequence of $\bE$ with $m_i = \min \supp_{\bE} (z_i)$, $i\!\in\!\N$,
then {\em $(z_i)_{i=1}^\infty$ is $C$-dominated by $(v_{m_i})_{i=1}^\infty$}.
Precisely, for all $(a_i)_{i=1}^\infty \subseteq \R$,
$$\Big\| \sum_{i=1}^\infty a_i z_i\Big\| \!\le\!
C \Big\| \sum_{i=1}^\infty a_i v_{m_i}\Big\|\ .$$
Similarly, {\em $(E_n)_{n=1}^\infty$ satisfies subsequential
$C$-$V$-lower estimates\/} if every such $(z_i)_{i=1}^\infty$
$C$-dominates $(v_{m_i})_{i=1}^\infty$.

We say that  {\em $(E_n)_{n=1}^\infty$ satisfies subsequential $V$-upper
estimates\/}  or  {\em subsequential $V$-lower
estimates\/}  if there exists a $C\ge 1$ so  that   $(E_n)_{n=1}^\infty$ satisfies subsequential $C$-$V$-upper
estimates\   or  subsequential $C$-$V$-lower
estimates, respectively.
\end{defn}

These are dual properties.
If $(v_i^*)_{i=1}^\infty$ are the biorthogonal functionals of
$(v_i)_{i=1}^\infty$ we define subsequential $V^*$-upper/lower estimates
to mean as above with respect to $(v_i^*)_{i=1}^\infty$.

\begin{prop}\label{P:OSZ1}\cite[Proposition 2.14]{OSZ1}
Let $Z$ have a bimonotone {\rm FDD} $(E_i)_{i=1}^\infty$ and let $V$ be a
Banach space with a normalized 1-unconditional basis $(v_i)_{i=1}^\infty$
with biorthogonal functionals $(v_n^*)_{n=1}^\infty$.
Let $1\!\le\! C\!<\!\infty$.
The following are equivalent.
\begin{itemize}
\item[a)] $(E_i)_{i=1}^\infty$ satisfies subsequential $C$-$V$-upper estimates
in $Z$.
\item[b)] $(E_i^*)_{i=1}^\infty$ satisfies subsequential $C$-$V^*$-lower
estimates in $Z^{(*)}$.
\end{itemize}
Moreover, the equivalence holds if we interchange ``upper'' with ``lower''
in {\rm a)} and {\rm b)}.
If the {\rm FDD} $(E_i)_{i=1}^\infty$ is not bimonotone the proposition
still holds but not with the same constants~$C$.
These changes depend upon $K(\bE,Z)$.
\end{prop}

Recall that $A\subseteq B_{Z^*}$ is {\em $d$-norming for $Z$}
($0\!<\! d \!\le\! 1$) if for all $z\!\in\! Z$,
$$d\|z\| \!\le\! \sup \{ |z^* (z)| : z^* \!\in\! A\}\ .$$

We  will need a characterization of subsequential $V$-upper
estimates obtained from  norming sets.

\begin{prop}\label{P:3.3}
Let $Z$ have an {\rm FDD} $\bE = (E_i)_{i=1}^\infty$ and let $V$ be a
Banach space with a normalized 1-unconditional basis $(v_i)_{i=1}^\infty$.
Let $0\!<\!d\!\le\! 1$ and let $A\subseteq B_{Z^*}$ be $d$-norming for $Z$.
The following are equivalent.
\begin{itemize}
\item[a)] $(E_i)_{i=1}^\infty$ satisfies subsequential $V$-upper estimates.
\item[b)] There exists $C\!<\!\infty$ so that for all $z^*\!\in\! A$ and any choice
of $k$ and $1\!\le\! n_1 \!<\! \cdots \!<\! n_{k+1}$ in $\N$,
$$\Big\| \sum_{i=1}^k \| z^* \circ P_{[n_i,n_{i+1})}^{\bE}\|
v_{n_i}^*\Big\| \!\le\! C\ .$$
Moreover, if $(E_i)_{i=1}^\infty$ is bimonotone, then $\a') \To \b') \To \b'')
\To \a'')$  where
\item[$\a')$]
$(E_i)_{i=1}^\infty$ satisfies subsequential $C$-$V$-upper estimates.
\item[$\b')$]
For every $x^*\!\in\! S_{Z^*}$ and any choice of $k$ and $1\!\le\! n_1 \!<\!n_2 \!<\! \cdots
\!<\! n_{k+1}$ in $\N$,
$$\Big\| \sum_{i=1}^k \| z^* \circ P_{[n_i,n_{i+1})}^{\bE} \|
v_{n_i}^*\Big\| \!\le\! C\ .$$
\item[$\b'')$]
For every $z^*\!\in\! A$ and any choice of $k$ and $1\!\le\! n_1 \!<\! \cdots \!<\! n_{k+1}$
in $\N$,
$$\Big\| \sum_{i=1}^k \| z^* \circ P_{[n_i,n_{i+1})}^{\bE} \|
v_{n_i}^*\Big\| \!\le\! C\ .$$
\item[$\a'')$]
$(E_i)_{i=1}^\infty$ satisfies subsequential $Cd^{-1}$-$V$-upper estimates.
\end{itemize}
\end{prop}

\begin{proof}
By renorming, we can assume that $(E_i)_{i=1}^\infty$ is bimonotone
and thus we need only prove the ``moreover'' statement.

\noindent $\a') \To \b')$ follows from Proposition~\ref{P:OSZ1}.
Indeed, $(z^* \circ P_{[n_i,n_{i+1})}^{\bE})_{i=1}^k$ is a block
sequence of $(E_i^*)$, whose sum has norm at most 1, and $\min
\supp_{\bE^*} (z^* \circ P_{[n_i,n_{i+1})}^{\bE})$ can be assumed
equal to $n_i$ by standard perturbation arguments.

\noindent $\b') \To \b'')$
is trivial.

\noindent $\b'') \To \a'')$.
Let $(z_i)_{i=1}^n$ be a normalized block sequence of $(E_i)$ with
$m_i = \min \supp_{\bE} (z_i)$ for $i\!\le\! n$.
Let $m_{n+1} = \max \supp_{\bE} (z_n) +1$.
Let $(a_i)_1^n \subseteq \R$ and choose $z^* \!\in\! A$ with
$$\Big| z^* \Big( \sum_{i=1}^n a_i z_i\Big) \Big|
\ge d\Big\| \sum_{i=1}^n a_i z_i \Big\|\ .$$
Thus,
\begin{align*}
\Big\| \sum_{i=1}^n a_i z_i\Big\|
 &\le d^{-1} \Big| \sum_{i=1}^n a_i z^* (z_i)\Big|\\
 &=
d^{-1} \Big| \sum_{i=1}^n a_i z^* \circ P_{[m_i,m_{i+1})}^{\bE}(z_i)\Big|\\
 &\le d^{-1} \sum_{i=1}^n |a_i|\, \|z^* \circ P_{[m_i,m_{i+1})}^{\bE}\|\\
 &=
d^{-1} \Big( \sum_{i=1}^n\| z^* \circ P_{[m_i,m_{i+1})}^{\bE}\| v_{m_i}^*\Big)
\Big( \sum_{i=1}^n |a_i| v_{m_i}\Big)
\le
C\, d^{-1} \Big\| \sum_{i=1}^n a_i v_{m_i}\Big\|,\  \text{by b'')}\ .
\end{align*}
\end{proof}

We recall some terminology concerning finite subsets of $\N$ which
can be found for example in \cite{AT} or \cite{OTW}.

\begin{defn}\label{D:pointwisetopology}
$[\N]^{<\omega}$ denotes the set of all finite subsets of $\N$ under
the {\em pointwise topology\/}, i.e., the topology it inherits as a subset
of $\{0,1\}^{\N}$ with the product topology.
Let $\cA \subseteq [\N]^{<\omega}$.
We say $\cA$ is
\begin{itemize}
\item[i)] {\em compact\/} if it is compact in the pointwise topology,
\item[ii)] {\em hereditary\/} if for all $A\!\in\! \cA$, if $B\subseteq A$
then $B\!\in\! \cA$,
\item[iii)] {\em spreading\/} if for all $A = (a_1,\ldots,a_n)\!\in\! \cA$
with $a_1 \!<\! a_2 \!<\!\cdots \!<\! a_n$ and all $B=(b_1,\ldots,b_n)\!\in\! [\N]^{<\omega}$
with $b_1 \!<\! b_2 \!<\!\cdots \!<\! b_n$ and $a_i \!\le\! b_i$ for $i\!\le\! n$, $B\!\in\!\cA$,
such a $B$ is called a {\em spread\/} of $A$,
\item[iv)] {\em regular\/} if $\{n\}\!\in\! \cA$ for all $n\!\in\!\N$ and $\cA$
is compact, hereditary and spreading.
\end{itemize}
\end{defn}
We note that if $\cA\subset [\N]^{<\omega}$ is relatively compact,
or equivalently if $\cA$ does not contain an infinite strictly
increasing chain,  then there is a regular family,
$\cB\subset[\N]^{<\omega}$, containing $\cA$.

\begin{defn}\label{D:A-admissible}
Let $\cA\subseteq [\N]^{<\omega}$ be a regular family.
A sequence of sets in $[\N]^{<\omega}$, $A_1 \!<\! A_2 \!<\! \cdots \!<\! A_n$
(i.e., $\max A_i \!<\! \min A_{i+1}$ for $i\!<\!n$) is called
{\em $\cA$-admissible\/} if $(\min A_i)_{i=1}^n \!\in\! \cA$.
\end{defn}

\begin{tsirelson}\label{D:tsirelson}
Let $\cA \subseteq [\N]^{<\omega}$ be a regular family of sets and let
$0\!<\!c\!<\!1$.
The Tsirelson space $\Tac$ is the completion of $c_{00}$ under the
norm $\|\cdot\|_{\cA,c}$ which is given, implicitly, by the equation
$$\|x\|_{\cA,c} = \|x\|_\infty\vee \sup \Big\{ \sum_{i=1}^n c
\|A_i x\|_{\cA,c} : n\!\in\! \N\ ,\
\text{ and }\ A_1 \!<\! \cdots \!<\! A_n\text{ is $\cA$-admissible}\Big\}\ .$$

Here $A_i x = x|_{A_i}$. The unit vector basis  $(t_i)$  of $c_{00}$
is always a shrinking  and 1-unconditional basis for $\Tac$.  If the
Cantor - Bendixson index of $\cA$ (c.f.  \cite{OTW} or \cite{AT}) is
at least $\omega$ then $\Tac$ does not contain any isomorphic copy
of $\ell_p$ or $c_0$, and hence $\Tac$ must also be reflexive as
every Banach space with an unconditional basis which does not
contain an isomorphic copy of $c_0$ or $\ell_1$ is reflexive.

If $\cA = S_\alpha$ is the $\alpha^{th}$-Schreier family of sets,
where $\alpha \!<\!\omega_1$, we denote $\Tac$ by $T_{c,\alpha}$.
For more on these spaces (see e.g., \cite{AT},
\cite{LTang},\cite{OSZ2} and the references therein). Let us recall
that,  for $n\in\N$, the spaces $T_{\alpha,c}$ and
$T_{\alpha^n,c^n}$ are naturally isomorphic
 (via the identity).
\end{tsirelson}

\begin{rem}
We will later use  the fact that if $X$ has an FDD $(E_i)_{i=1}^\infty$
satisfying subsequential $\Tac$-upper estimates for some regular
family $\cA$, then $(E_i)_{i=1}^\infty$ is shrinking.
Indeed every normalized block sequence of $(E_i)_{i=1}^\infty$ must then
be weakly null, since it is dominated by a weakly null sequence.
This is equivalent to $(E_i)_{i=1}^\infty$ being shrinking.
\end{rem}

Our embedding theorems, \ref{thm:FOSZ} and \ref{thm:OSZ2} below, refer
to the Szlenk index, $S_z(X)$, \cite{Sz}.
If $X$ is separable then $S_z(X)$ is an ordinal with $S_z(X)\!<\!\omega_1$
if and only if $X^*$ is separable.
Also $S_z(T_{c,\alpha}) = \omega^{\alpha\cdot\omega}$
 \cite[Proposition 7]{OSZ2}.
If $S_z(X) \!<\!\omega_1$ then $S_z(X) = \omega^\beta$ for some $\beta \!<\!\omega_1$.
Much has been written on the Szlenk index
(e.g.,  see \cite{AJO}, \cite{B2}, \cite{FOSZ}, \cite{G}, \cite{GKL},
\cite{JO}, \cite{L}, \cite{OSZ2}).

\begin{thm}\label{thm:FOSZ}\cite[Theorem 1.3]{FOSZ}
Let $\alpha \!<\!\omega_1$ and let $X$ be a Banach space with separable dual.
The following are equivalent.
\begin{itemize}
\item[a)] $S_z(X) \!\le\! \omega^{\alpha\cdot\omega}$.
\item[b)] $X$ embeds into a Banach space $Z$ having an {\rm FDD} which
satisfies subsequential $T_{c,\alpha}$-upper estimates, for some $0\!<\!c\!<\!1$.
\end{itemize}
\end{thm}

\begin{thm}\label{thm:OSZ2}\cite[Theorem A]{OSZ2}
Let $\alpha \!<\!\omega_1$ and let $X$ be a separable reflexive Banach space.
The following are equivalent.
\begin{itemize}
\item[a)] $S_z(X) \!\le\! \omega^{\alpha\cdot\omega}$ and
$S_z(X^*) \!\le\! \omega^{\alpha\cdot\omega}$.
\item[b)] $X$ embeds into a Banach space $Z$ having an {\rm FDD} which
satisfies both subsequential $T_{c,\alpha}$-upper estimates and
subsequential $T_{c,\alpha}^*$-lower estimates, for some
$0\!<\!c\!<\!1$.
\end{itemize}
\end{thm}

We note that the upper and lower estimates in both theorems are with
respect to the unit vector basis $(t_i)$ of $T_{c,\alpha}$ and its
biorthogonal sequence $(t_i^*)$, a basis for $T_{c,\alpha}^*$.

In order to use Theorem~\ref{thm:FOSZ} in our proof of Theorem~A, we need
to reformulate what it means for an FDD for $X$ to satisfy subsequential
$T_{c,\alpha}$-upper estimates in terms of the functionals in $X^*$.
We first need some more terminology.

\begin{defn}\label{D:c-decomposition}
Let $\bE = (E_i)_{i=1}^\infty$ be an FDD for a space $X$ and let $0\!<\!c\!<\!1$.
Let $x\!\in\! c_{00}(\oplus_{i=1}^\infty E_i)$.
A block sequence of $\bE$, $(x_1,\ldots,x_\ell)$, is called a
{\em $c$-decomposition of $x$} if
\begin{equation}\label{eq:c-decomposition}
x= \sum_{i=1}^\ell x_i\ \text{ and, for every $i\!\le\! \ell$, either }\
|\supp_{\bE} (x_i)| =1\ \text{ or }\ \|x_i\| \!\le\! c\ .
\end{equation}
\end{defn}

Clearly every such $x$ has a $c$-decomposition.
The {\em optimal $c$-decomposition of $x$} is defined as follows.
Set $n_1 = \min \supp_{\bE}(x)$ and assume $n_1 \!<\! n_2\!<\!\cdots \!<\! n_j$
have been defined.
Let
\begin{equation*}
n_{j+1} = \begin{cases}
n_j +1\ ,\ \text{ if }\ \|P_{n_j}^{\bE} (x)\| >c\ ,\\
\noalign{\vskip4pt}
\min \{ n : \|P_{[n_j,n]}^{\bE} (x)\| >c\}\ ,\ \text{ if }\
\|P_{n_j}^{\bE} (x)\| \!\le\! c\
\text{ and the ``min'' exists,}\\
\noalign{\vskip4pt}
1+\max \supp_{\bE} (x)\ ,\ \text{ otherwise.}
\end{cases}
\end{equation*}
There will be a smallest $\ell$ so that $n_{\ell+1} = 1+\max \supp_{\bE}(x)$.
We then set for $i\!\le\! \ell$,
$x_i = P_{[n_i,n_{i+1})}^{\bE} (x)$.
Clearly $(x_i)_{i=1}^\ell$ is a $c$-decomposition of $x$.
Moreover, and this will be important later, if
$(E_i)$ is bimonotone and
$j\!\le\! \lfloor \sfrac{\ell}2\rfloor$,
then $\|x_{2j-1} + x_{2j}\| > c$.

Let $\cA \subseteq [\N]^{<\omega}$ be regular.
We say that the {\em FDD $(E_i)_{i=1}^\infty$ for $X$ is $(c,\cA)$-admissible
in $X$} if every $x\!\in\! S_X\cap c_{00} (\oplus_{i=1}^\infty E_i)$ has an
$\cA$-admissible $c$-decomposition, $(x_i)_{i=1}^k$, where
$(\supp_{\bE} (x_i))_1^\ell$ is {\em $\cA$-admissible}, i.e.,
$(\min \supp_{\bE}(x_i))_{i=1}^\ell \!\in\! \cA$.

\begin{thm}\label{P:bimonotone}
Let $\bE = (E_i)_{i=1}^\infty$ be a bimonotone {\rm FDD} for a
Banach space $X$. The following statements are equivalent.
\begin{itemize}
\item[a)] $(E_i)$ is shrinking.
\item[b)] For all $0<c<1$ there exists a regular family
$\cA\subset[\N]^{<\omega}$ so that every $x^*\in B_{X^*}\cap
\coo(\oplus_{i=1}^\infty E_i^*)$ has an optimal $\cA$-admissible
$c$-decomposition.
\item[c)] There exists $D\subset B_{X^*}\cap\coo(\oplus_{i=1}^\infty
E_i^*)$, $0<c<d\leq1$ and a regular family
$\cA\subset[\N]^{<\omega}$, so that $D$ is $d$-norming for $X$, and
every $x^*\in D$ admits an $\cA$-admissible $c$-decomposition.
\item[d)] There exists $\alpha\!<\!\omega_1$, $0\!<\!c\!<\!1$, $1\!\leq\!C$,
and a subsequence $(t_{m_i})_{i=1}^\infty$ of the unit vector basis
for $T_{c,\alpha}$, so that $(E_i)_{i=1}^\infty$ satisfies
subsequential $C-(t_{m_i})_{i=1}^\infty$ upper estimates.
\end{itemize}
\end{thm}

\begin{proof}
$a)\Rightarrow b)$.  Assume $b)$ fails for some $0<c<1$.  Then the
set $$\{(\min\supp_{E^*}(x_i^*))_{i=1}^n:\,(x_i^*)_{i=1}^n\textrm{
is the optimal $c$-decomposition of some }x^*\in
B_{X^*}\cap\coo(\oplus_{i=1}^\infty E_i^*) \}$$ is not
relatively compact in $[\N]^{<\omega}$.  This yields a sequence
$(n_i)_{i=1}^\infty\in[\N]^\omega$ so that for all $N\in\N$, there
exists $x^*(N)\in B_{X^*} \cap c_{00}(\oplus_{i=1}^\infty
E_i^*)$, with an optimal $c$-decomposition $(x^*_i(N))_{i=1}^
{\ell(N)}$ so that $\min\supp_{E^*}(x_i^*(N))=n_i$ for all $i\leq
N$. After passing to a subsequence, we may assume that
$\lim_{N\rightarrow\infty} x_i^*(N)=x_i^*$ for some $x^*_i\in
B_{X^*}\cap\coo(\oplus_{i=1}^\infty E_i^*) $ with
$\supp(x_i^*)\subset [n_i ,n_{i+1})$ for all $i\in\N$.  We have that
$\|x^*_{i}(N)+x^*_{i+1}(N)\|\geq c$ for all $N\in\N$ and $1\leq i<
\ell(N)$, and hence $\|x^*_{i}+x^*_{i+1}\|\geq c$ for all $i\in\N$.
Furthermore, $\|\sum_{i=1}^N x_i^*(N)\|\leq\|\sum_{i=1}^{\ell(N)}
x_i^*(N)\|\leq1$ for all $N\in\N$, and hence
$\sup_{N\in\N}\|\sum_{i=1}^N x_i^*(N)\|\leq1$.  We conclude that
$(x_i^*)$ is not boundedly complete, and hence $(E_i)_{i=1}^\infty$
is not shrinking.

$b)\Rightarrow c)$ is trivial.

$c)\Rightarrow d)$.  Let $D$, $0<c<d\leq 1$, and $\cA$ be as in
$c)$. We define $$\cB=\{{n}\cup B_1\cup B_2:\,n\in\N,\, B_1,B_2\in
A\}\cup\{\emptyset\}.$$ It is easily checked that $\cB = \cB_{\cA}$
is regular. Let $(t_i)_{i=1}^\infty$ be the unit vector basis of
$T_{c/d,\,\cB}$. We will prove, by induction on $s\!\in\! \N$, that
if $(x_i)_{i=1}^k$ is a normalized block sequence of $\bE$ with
finite length and $|\supp_{\bE} (\sum_{i=1}^k x_i)| \!\le\! s$, then
for all $(a_i)_1^k\subseteq\R$,
\begin{equation}\label{eq:3.2}
\Big\| \sum_{i=1}^k a_i x_i\Big\| \!\le\! c^{-1} \Big\| \sum_{i=1}^k
a_i t_{\min\supp_{\bE}(x_i)}\Big\|_{T_{c/d,\, \mathcal B}}\ .
\end{equation}

This is trivial for $s=1$ and also clear for $k=1$, so we may assume
$k>1$. Assume it holds for all $s'\!\le\! s$. Let $(x_i)_{i=1}^k$ be
a normalized block sequence of $\bE$ with $|\supp_{\bE}
(\sum_{i=1}^k x_i)| = s+1$. Let $m_i = \min \supp_{\bE} (x_i)$ for
$i\!\le\! k$ and set $m_{k+1} = 1+\max \supp_{\bE} (x_k)$. Let
$(a_i)_{i=1}^k\subseteq \R$ and $c/d \!<\! \rho \!<\!1$ be
arbitrary. Since $D$ is $d$-norming for $X$, there exists
$x^*\!\in\! D$ with
$$\Big| x^* \Big( \sum_{i=1}^k a_i x_i\Big) \Big|
\ge \rho d\Big\|  \sum_{i=1}^k a_i x_i\Big\|\ .$$ Let $\tilde x^* =
P_{[m_1,m_{k+1})}^{\bE^*} (x^*)$ where $\bE^* =
(E_j^*)_{j=1}^\infty$  is the FDD for $X^{(*)}$. By the
bimonotonicity of $\bE$, $\|\tilde x^*\| \!\le\!1$ and also
$\|\tilde x^* (\sum_{i=1}^k a_i x_i)\| \ge \rho d\|\sum_{i=1}^k a_k
x_i\|$. Furthermore, since $x^*$ admits an $\cA$-admissible
$c$-decomposition, so does $\tilde x^*$. Let $(x_i^*)_{i=1}^\ell$ be
an $\cA$-admissible $c$-decomposition of $\tilde x^*$ and let $n_i =
\min\supp_{\bE^*} (x_i^*)$ for $i\!\le\! \ell$. Thus
$(n_i)_{i=1}^\ell \!\in\!\cA$.

If $\ell=1$, then $\tilde x^* \!\in\! E_j^*$    for some $j$ and so
\begin{align*}
\Big\|\sum_{i=1}^k a_i x_i\| &\le (\rho\, d)^{-1} \Big| \tilde x^*
\Big( \sum_{i=1}^k a_i x_i\Big)\Big|
 \le
(\rho\, d)^{-1}  |a_j| \\
&\le (\rho\, d)^{-1} \Big\| \sum_{i=1}^k a_i t_{m_i}\Big\| \le
c^{-1} \Big\| \sum_{i=1}^k a_i t_{m_i}\Big\|\ , \ \text{ so
\eqref{eq:3.2} holds.}
\end{align*}
\ If $\ell>1$, we proceed as follows. Define
\begin{align*}
B_1 & = \{ m_i : i\!\le\! k \text{ and  there exists $j\!\le\! \ell$
with }
m_i\!\le\! n_j \!<\! m_{i+1}\}\ ,\\
B_2 & = \{ m_{i+1} : i\!\le\! k\ \text{ and }\ m_i\!\in\! B_1\}\ ,
\end{align*}
and let $n = \min (B_1)$. Then $B\equiv B_1 \cup B_2 = \{n\} \cup
(B_1\setminus\{n\}) \cup B_2\!\in\! \cB_{\cA}$. Indeed $B_2
\!\in\!\cA$ since it is a spread of a subset of
$(n_j)_{j=1}^\ell\!\in\!\cA$, by the definition of $B_1$. Similarly
$B_1\setminus \{n\}\!\in\! \cA$.

Write $B = \{m_{b_j} :j\!\le\! \ell'\}$ where $b_1 \!<\! b_2
\!<\!\cdots \!<\! b_{\ell'}$. Set $m_{b_{\ell'+1}} = m_{k+1}$. Since
$k>1$, $|\supp_{\bE} (\sum_{i=b_j}^{b_{j+1}-1} x_i)|\!\le\! s$, for
$j\!\le\! \ell'$, and our induction hypothesis applies to such
blocks. Moreover, if $b_{j+1} \ne b_j +1$ for some $j\!\le\! \ell'$,
then there is at most one $x_t^*$ whose support is not disjoint from
$\oplus_{i=m_{b_j}}^{m_{b_{j+1}-1}} E_i^*$, since no $n_i$ can
satisfy $m_{b_j} \!<\! n_i \!<\! m_{b_{j+1}}$. In addition,
$|\supp_{\bE^*} (x_t^*)| >1$ in this case, and so $\|x_t^*\| \!\le\!
c$ which yields
$$\Big| \tilde x^* \Big( \sum_{i=b_j}^{b_{j+1}-1} a_i x_i\Big)\Big|
\!\le\! c\,\Big\| \sum_{i=b_j}^{b_{j+1}-1} a_i x_i\Big\|\ .$$

We obtain for $I = \{j\!\le\!\ell': b_{j+1} \ne b_j+1\}$ and $J =
\{1,\ldots,\ell'\}\setminus I$,
\begin{align*}
\rho\,d\Big\| \sum_{i=1}^k a_i x_i\Big\|
&\le\Big| \tilde x^*\Big( \sum_{i=1}^k a_i x_i\Big) \Big|\\
&\le \Big| \sum_{j \in I} \tilde x^* \Big( \sum_{i=b_j}^{b_{j+1}-1}
a_ix_i\Big)\Big|
+ \Big| \sum_{j \in  J} \tilde x^* (a_{b_j} x_{b_j})\Big|\\
&\le \sum_{j \in I} c\, \Big\|\sum_{i=b_j}^{b_{j+1}-1} a_i x_i\Big\|
+ \sum_{j \in J} |a_{b_j}|\\
&\le \sum_{j \in  I} \Big\|\sum_{i=b_j}^{b_{j+1}-1} a_i
t_{m_i}\Big\| + \sum_{j \in  J} \|a_{b_j} t_{m_{b_j}}\|\ ,
\ \text{ by the induction hypothesis,}\\
& = \frac{d}c \sum_{j=1}^{\ell'} \frac{c}d
\Big\|\sum_{i=b_j}^{b_{j+1}-1} a_i t_{m_i}\Big\| \le \frac{d}c
\Big\| \sum_{i=1}^k a_i t_{m_i}\Big\|\ ,
\end{align*}
by definition of the norm for $T_{c/d,\, \cB_{\cA}}$. So
$$\rho\,c\Big\|\sum_{i=1}^k a_i x_i\Big\|
\le \Big\|\sum_{i=1}^k a_i t_{m_i}\Big\|\ .$$ Since $\rho \!<\!1$
was arbitrary  this proves \eqref{eq:3.2}.  Now the set $\cB$ is
regular, so its Cantor-Bendixson index $CB(\cB)$ is less than
$\omega_1$. By Proposition 3.10 in \cite{OTW}, if $\alpha<\omega_1$
is such that $CB(\cB)\leq\omega^\alpha$ then there exists
$(m_i)_{i=1}^\infty\in[\N]^\omega$ such that $\{(m_i)_{i\in
F}:F\in\cB\}\subset S_\alpha$.  It follows, from (\ref{eq:3.2}) that
$(E_i)$ satisfies subsequential $c^{-1}-(t_{m_i})_{i=1}^\infty$
upper estimates, where $(t_i)_{i=1}^\infty$ is the unit vector basis
of $T_{c/d,\alpha}$.

$d)\Rightarrow a)$ is immediate since $(t_{m_i})$ is weakly null.

\end{proof}

\begin{rems}\label{R:bimonotone}
In Theorem~\ref{P:bimonotone}, if the FDD $(E_i)$ for $X$ is not
bimonotone, then the Proposition holds with slight modification. Let
$K$ be the projection constant of $(E_i)$.  The hypothesis
``$0<c<d$'' in $c)$ should be changed to ``$0<c<d/K$''. This is seen
by renorming $X$, in the standard way, so that $(E_i)$ is
bimonotone:
$$|||x|||=\sup_{m\leq n}\|P^\bE_{[m,n]}\|.
$$
Then $D$ becomes $d/K$-norming for $(X,|||\cdot|||)$.  Furthermore,
(\ref{eq:3.2}) becomes valid for $(X,||\cdot||)$ with $c^{-1}$
replaced by $Kc^{-1}$.
\end{rems}

  It is worth noting that Proposition \ref{P:bimonotone} yields, as
  a corollary, the following less exact version of Theorem
  \ref{thm:FOSZ}.  A similar version of Theorem \ref{thm:OSZ2} would
  also follow.

\begin{cor}
Let $X$ be a Banach space with $X^*$ separable. Then there exists
$\alpha<\omega_1$ and $0<c<1$ so that $X$ embeds into a space $Y$,
with an FDD $(F_i)$ satisfying subsequential $T_{c,\alpha}$-upper
estimates.
\end{cor}
\begin{proof}
By Zippin's theorem \cite{Z2}, we may embed $X$ into a space $Z$
with a shrinking FDD $(E_i)$.  By Theorem \ref{P:bimonotone}
$d)$, we obtain the result, except that the estimates are with
respect to $(t_{m_i})$.  We expand the FDD by inserting the basis
vectors $(t_j)_{j\in(m_{i-1},m_i)}$ between $E_{i-1}$ and $E_{i}$ to
obtain the desired FDD in a subspace of $Z\oplus T_{c,\alpha}$.
\end{proof}

Using Proposition  \ref{P:1.5}  we can derive from Theorem \ref{P:bimonotone}  the following sufficient and necessary condition for the dual
 of a Bourgain-Delbaen space  to be isomorphic to $\ell_1$.

 \begin{cor}\label{C:3.7} Let $Y$ be the Bourgain-Delbaen space associated to a Bourgain-Delbaen sequence $(\Delta_n)$
  satisfying  condition \eqref{E:1.1a}  for some $\theta<1/2$ (and thus the conclusion of  Proposition \ref{P:1.4} with $M\le \max(1/(1-2\theta), 2)$ )
 and  let $\bF=(F_j)$ be the FDD of $Y$ as introduced in Section \ref{S:2} and $\bF^*=(F_j^*)$.
 Define
 $$\cC=\Big\{\cuts(\gamma):\gamma\in\bigcup_{n=1}^\infty \Delta_n\Big\}.$$
 Then $\bF$ is shrinking (and thus $Y^*$ is isomorphic to $\ell_1$) if
 $\cC$ is compact, or equivalently, if $\cC$ does not contain an infinite strictly increasing chain.
\end{cor}
\begin{proof}
  Indeed, assuming \eqref{E:1.1a},  in  the analysis of $\gamma\in\Gamma$
  $$e^*_\gamma=
 \sum_{j=1}^{a}  \alpha_j d^*_{\xi_j} +  \beta_j P^*_{(p_{j-1},p_j)}(b^*_j) .
$$
all the $\beta_j$'s are at most $\theta$, except the ones for which
the support of $P^{\bF^*}_{(p_{j-1},p_j)}(b^*_j)$ (with respect to
$\bF^*$) is at most a singleton. Therefore the analysis  of $\gamma$
represents a $c$-decomposition of $e^*_\gamma$  and, thus,
Theorem  \ref{P:bimonotone} yields that $\bF$ is shrinking.
 \end{proof}

\section{The proof of Theorem A}\label{S:4}

Let $X$ be a separable Banach space. We will follow the  generalized BD construction in Section
\ref{S:2} to embed $X$ into a $\cL_\infty$ space $Y$. Since $X$ can be embedded into a space with basis
 (for example $C[0,1]$), we can assume that $X$ has an FDD,  which we denote by $\bE=(E_i)$, and after a renorming, if
 necessary, we can assume that $\bE$ is bimonotone.
  If $X^*$ is separable then we can assume
  that $\bE$ is shrinking by \cite{Z2}.

 The  Bourgain-Delbaen space $Y$, which we construct to contain $X$, will have $Y^*$ isomorphic
 to $\ell_1$, in the case that $X^*$ is separable.

 To begin we fix $0<c\leq 1/16$ and choose $0<\vp<c$, and  $(\vp_i)_{i=1}^\infty \subset (0,\vp)$ with $\vp_i\downarrow 0$ so that
\begin{align}\label{eq:4.1}
&\sum_{i=1}^\infty \vp_i \!<\! \frac{\vp}8\quad \text{ and} \sum_{i>n} \vp_i \!<\! \frac{\vp_n}2\ \text{ for all }\ n\!\in\! \N\ .
\end{align}

Next, for $i\!\in\! \N$, we choose $ R_i\subset (0,1]$ and $ \tilde
A_i^*\subseteq S_{E_i^*}$ to be $\vp_i/8$, dense in their respective
supersets,  with $1\in R_i$ for all $i\in\N$.
 We then choose an appropriate countable  subset,  $D\subset B_{X^*}\cap \coo(\oplus E^*_i)$, which norms $X$.

 \begin{lem}\label{L:4.1}
 There exists  a  set $D\subset \big(B_{X^*}\setminus \frac12 B_{X^*}\big)\cap \coo(\oplus E^*_i)$  with the following properties.
 \begin{enumerate}
 \item[a)]   $A^*_m:=D\cap E^*_m=\frac1{1+\vp/4} \tilde A^*_m$, for $m\in \N$.
 \item[b)]
   $D\cap\big(\oplus_{j=m}^n E^*_j\big)$  is finite,
   and $(1-\vp)$-norms the elements of $\oplus_{j=m}^nE_j$, for all $m<n$ in $\N$.
 \item[c)]   Every $x^*\in D$   can be written as
 $x^*=\sum_{i=1}^\ell r_i x^*_i$, where $(r_1x^*_1,\ldots,r_\ell x^*_\ell)$, is a  $c$-decomposition of $x^*$  and $x^*_i\in D$,  and $r_i\in R_{\max\supp(x^*_i)}$, for $i=1,\ldots \ell$.
 Moreover
 $$(\supp(x^*_i))_{i=1}^\ell\in\left\{  (\supp(z^*_i))_{i=1}^{\ell} : \begin{matrix}&(z^*_i)_{i=1}^\ell \text{ is the optimal $\frac{c}{1+\vp/4}$-decomposition}\\
                                                                                          &\text{ of some $z^*\in B_{X^*}\cap \coo\big(\oplus_{j=1}^\infty E^*_j\big)$}\end{matrix}\right\}.$$
                                                                                         \end{enumerate}
If $(E_i)$ is 1-uncondtional in $X$  then (a) and (b) can be
replaced by
\begin{enumerate}
\item[a')] $A^*_m:=D\cap E_m=\tilde A^*_m$, for $m\in \N$.
\item[b')]
   $D\cap\big(\oplus_{j\in B} E^*_j\big)$  is finite,
   and $(1-\vp)$-norms the elements of $\oplus_{j\in B}E_j$, for all finite $B\subset\N$.
\end{enumerate}
 \end{lem}
 For  $D$ as in  Lemma \ref{L:4.1} and
 each $x^*\in D$ we pick  such a $c$-decomposition  $(r_1x^*,r_2x^*_2,\ldots r_\ell x^*)$ and call it the {\em special $c$-decomposition of $x^*$.}
 If $x^*\in A^*_j=D\cap E^*_j$, we let $(x^*)$   be its own special $c$-decomposition.

 \begin{proof} We abbreviate $\supp_{\bE^*}(\cdot)$ by $\supp(\cdot)$, and we abbreviate $\textrm{ran}_{\bE^*}(\cdot)$ by $\textrm{ran}(\cdot)$.
Define
$$H= \frac1{1+\vp/4}\left\{
\frac{\sum\limits_{i=m}^n a_i x_i^*}{\|\sum\limits_{i=m}^n a_i
x_i^*\|} :  m\!\leq\! n\ ,\ a_i \!\in\!  R_i\ \text{ and }\
x_i^*\!\in\! \tilde A_i^* \ \text{ for }\ i\!\in\! [m,n]\right\}\\
.$$ We note the following properties of $H$.
\begin{align}
&H\ \text{ is countable.}\label{eq:4.3} \\
&H\cap \oplus_{i=1}^n E_i^* \ \text{ is finite for all }\ n\!\in\! \N\ .\label{eq:4.4}\\
&H\cap \oplus_{i=m}^n E_i^*\ \text{ $(1-\vp)$-norms  $\oplus_{i=m}^n E_i$, for all $m\le n$ in $\N$.}\label{eq:4.5} \\
&\text{If $x^*\!\in\! H$ and }\supp(x^*)\cap [m,n]\ne\phi ,\
  m\!\le\!n,\text{ then } \frac{P_{[m,n]}^{\bE^*}(x^*)}{\|P_{[m,n]}^{\bE^*}(x^*)\|}
\!\in\!(1\!+\!\vp/4) H. \label{eq:4.6}\end{align}

Set $H_n=\{h\in H: |\ran(h)|=n\}$ and thus $H=\bigcup_{n=1}^\infty
H_n$. For each $n\in\N$ we will  inductively define for $h\in H_n$,
an element $\tilde h\in \big(B_{X^*}\setminus\frac12B_{X^*}\big)\cap
\coo(\oplus_{j=1}^\infty E^*_i)$.  We then set $D_n=\{\tilde{h}:h\in
H_n\}$ and $D=\cup_{n\in\N}D_n$.

If $h\in H_1$, let $\tilde{h}=h$.  Let $n>1$ and assume that $D_m$
has been defined for all $m<n$.  Let $h\in H_n$ and
$(z_1^*,\ldots,z_\ell^*)$ be the optimal $c/(1+\vp/4)$-decomposition
of $h$. Note that $\ell\geq2$ since $n>1$ and $\|h\|=1/(1+\vp/4)$.
We write the decomposition as $$(s_i
h_i)_{i=1}^\ell=\left(\|z_i^*\|(1+\vp/4)\frac{z_i^*}{(1+\vp/4)\|z_i^*\|}\right)_{i=1}^\ell.$$
By the definition of $H$, $\|z_i^*\|\leq1/(1+\vp/4)$ and so
$0<s_i=\|z_i^*\|(1+\vp/4)\leq1$ for $i\leq\ell$.  If $h_i\not\in H_1$,
then $\|s_i h_i\|=\|z_i^*\|\leq c/(1+\vp/4)$ and so $s_i\leq c$.

For $i\leq\ell$, choose $r_i\in R_{\max\supp(h_i)}$ with
$|r_i-s_i|\leq\vp_{\max\supp(h_i)}/4$ and $r_i\leq c$ if $h\not\in
H_1$.  We define $\tilde{h}=\sum_{i=1}^\ell r_i \tilde{h}_i$.  By
induction, we will verify the following.
\begin{equation}\label{E:4.6}\supp(\tilde{h})=\supp(h)
\end{equation}
\begin{equation}\label{E:4.7}\|\tilde{h}-h\|\leq\sum_{j\in\supp(\tilde{h})}\vp_j
\end{equation}
\begin{equation}\label{E:4.8}(r_1\tilde{h}_1,..,r_\ell\tilde{h}_\ell)\textrm{
is a $c-$decomp of }\tilde{h},\textrm{ with $r_i\kin
R_{\max\supp(\tilde{h}_i)}$ and
}\tilde{h}_i\kin\cup_{m<n}D_m\textrm{, if }n\!>\!1.
\end{equation}
The condition (\ref{E:4.6}) is clear.  To verify (\ref{E:4.7}) we
note that if $h_i\in H_1$, then
$$\|r_i\tilde{h}_i-s_i
h_i\|\le|r_i-s_i|<\vp_{\max\supp(\tilde{h}_i)}/4.
$$
If $h_i\not\in H_1$, by the induction hypothesis,
\begin{align*}\|r_i\tilde{h}_i-s_i h_i\|&\leq
\|r_i(\tilde{h}_i-h_i)\|+\|(r_i-s_i)h_i\|
\leq
c\sum_{j\in\supp(\tilde{h}_i)}\vp_j+\vp_{\max\supp({h}_i)}/4\leq\sum_{j\in\supp(\tilde{h}_i)}\vp_j.
\end{align*}
Thus $\|h-\tilde{h}\|\leq\sum_{i=1}^\ell\|r_i\tilde{h}_i-s_i
h_i\|<\sum_{j\in\supp(\tilde{h})}\vp_j$, which proves (\ref{E:4.7}). \eqref{E:4.8} holds by construction.
Equation (\ref{E:4.7}) now yields,
\begin{align*}
1/2&\leq1/(1+\vp/4)-\sum_{j\in\supp(\tilde{h})}\vp_j\leq\|h\|-\|h-\tilde{h}\|\\
&\leq\|\tilde{h}\|\leq\|h\|+\|h-\tilde{h}\|\leq1/(1+\vp/4)+\sum_{j\in\supp(\tilde{h})}\vp_j\leq1.
\end{align*}
Thus $D\subset
B_{X^*}\setminus\frac{1}{2}B_{X^*}$.  Properties $a),b),$ and $c)$
of $D$ follow from $(\ref{E:4.6}),(\ref{E:4.7})$, and
$(\ref{E:4.8})$.

If $(E_i)$ is 1-unconditional, as defined, we instead begin with
$$H= \left\{
\frac{\sum\limits_{i\in B} a_i x_i^*}{\|\sum\limits_{i\in B}^n a_i
x_i^*\|} :  \emptyset\neq B\subset\N,\,|B|<\infty  ,\, a_i \!\in\!
R_i\ \text{ and }\
x_i^*\!\in\! \tilde A_i^* \ \text{ for }\ i\!\in\! B\right\}\\
.$$ We then follow the above construction, similarly without the
$(1+\vp/4)$-factors.  These were necessary to ensure that the
$\tilde{h}_j$'s were in $B_{X^*}$.
\end{proof}

Next we define $\Gamma$ and a certain partial order on $\Gamma$ and use
that to define the $\Delta_n$'s.
$$\Gamma = \left\{ (r_1x_1^*,\ldots, r_jx_j^*) :\begin{matrix} \text{$j\ge 1$ and there exists $y^*\!\in\! D$
so that $(r_1x_1^*,\ldots,r_jx_j^*)$}\\
\text{are the first $j$ elements of the
 special $c$-decomposition of $\ds y^*$}\end{matrix}\right\}\,.$$

From Theorem \ref{P:bimonotone}  and Lemma   \ref{L:4.1} we deduce  for $\cG\!=\!\big\{\{\min\supp(x_j^*): j\!\le\!\ell\}: (r_1x^*_1,\ldots r_\ell x^*_\ell)\!\in\!\Gamma\big\}$
\begin{equation}\label{E:4.3a}  \text{ $(E_i)$ is shrinking in $X$ $\iff$  $\cG$ is compact.}
\end{equation}

We first define an order on the bounded intervals in $\N$ by
$[n_1,n_2] < [m_1,m_2]$ if $n_2 \!<\! m_2$ or $n_2 = m_2$ and $n_1> m_1$.
It is not hard to see that this is a well ordering.
It is instructive to list the first few elements in increasing order
(we let $[n,n]=n$):
$$(I_n)_{n=1}^\infty =
(1,\, 2,\, [1,2],\, 3,\, [2,3],\, [1,3],\, 4,\, [3,4],\, [2,4],\, [1,4],\, 5\,
\ldots)$$
If $\gamma = (x_1^*,\ldots, x_\ell^*)\!\in\! \Gamma$ we let
$$\ran_{\bE^*}\Big(\sum_{i=1}^\ell x_i^*\Big) \equiv \ran_{\bE^*} (\gamma)
\quad\text{ and }\quad
\supp_{\bE^*} \Big(\sum_{i=1}^\ell x_i^*\Big) \equiv \supp_{\bE^*} (\gamma)\ .$$

For $\gamma \!\in\!\Gamma$ we define {\em the rank of $\gamma$} by
$\rk(\gamma)=n$ if $\ran\supp_{\bE^*}(\gamma)=I_n$.
We then define a partial order ``$\le$'' on $\Gamma$ by $\gamma \!<\!\eta$
if $\rk(\gamma) \!<\! \rk_{\bE^*}(\eta)$.
 If $\rk(\gamma)=\rk(\xi)$ and $\gamma\not=\eta$ we say that $\gamma$ and $\eta$ are incomparable. We next define an important subsequence $(m_j)_{j=1}^\infty$ of $\N$.
For $j\!\in\!\N$ let $m_j = \rk(x^*)$ for $x^* \!\in\! A_j^*$.  Thus $m_1 =1$, $m_2=2$, $m_3=4$ and more generally $m_{j+1} = m_j + j$.
 Note that
\begin{equation}\label{eq:4.12}
\text{for $\gamma\!\in\!\Gamma$, $i_0 = \max\supp_{\bE^*}(\gamma)$ if and only if }
m_{i_0} \!\le\! \rk(\gamma) \!<\!m_{i_0+1}\ .
\end{equation}

The following proposition is easily verified.
\begin{prop}\label{P:partialorder}
``$\le$'' is a partial order on $\Gamma$.
Furthermore,
\begin{itemize}
\item[a)] Every  natural number  is the rank of some element of
$\Gamma$ and the set of all such elements is finite.
\item[b)] If $j\!\in\! \N$ and
 $(z^*) \!\in\!  \{\gamma:\rk(\gamma)=m_j\}= \{(rx^*)\in \Gamma: r\in R_j, x^*\in A^*_j\}$ , then
\begin{align*}
&\qquad\quad\,\,\,\,\{\gamma \!\in\!\Gamma : \gamma \!<\! z^*\}
= \{ \gamma \!\in\! \Gamma : \max\supp_{\bE^*} (\gamma) \!<\! j\}\qquad
\text{ and}\\
& \begin{matrix}\{\gamma\!\in\!\Gamma : \gamma \!>\!(z^*)\}
  = \{\gamma \!\in\! \Gamma :\max\supp_{\bE^*} (\gamma) \!\ge\!j\textrm{ and }\supp_{E^*}(\gamma)\neq\{j\}\}\ .\end{matrix}
\end{align*}
\end{itemize}
\end{prop}
\begin{proof}  Lemma \ref{L:4.1}  (b) implies
  that for any $n$
there must be some $\gamma\in\Gamma$ of rank $n$,  and
if we  let $s<t$, so that $I_n=(s,t]$,
  then
  \begin{align*}
\#\{\gamma \kin\Gamma: \,&\rk(\gamma)\!=\!n\}
 \le     \sum_{\ell=1}^{ t-s}\sum_{s=t_0<t_1<\ldots t_\ell=t}\,\prod_{j=1}^\ell \#R_{t_j} \cdot
\# D\cap (\oplus_{j=t_{j-1}}^{t_j} E^*_j),
\end{align*}
which yields (a).  (b) follows easily from the definition of our
partial order.
\end{proof}

For $n\kin\N$, set $\Delta_n=\{\gamma\in\Gamma:\rk(\gamma)=n\}$. We
will next define $c^*_\gamma$ for $\gamma\in\Gamma$ (thus also
defining $e^*_\gamma=c_\gamma^*+d^*_\gamma$). Following this we will
show how the $\Delta_n$'s can be recoded  to fit into
 the framework of Section \ref{S:2}. To begin,
 \begin{enumerate}
\item[i)] we let
  $c^*_\gamma=0$  if $\rk(\gamma)\in\{m_j:j\in\N\}$  (thus, in particular,  $c^*_\gamma=0$ if $\gamma\in \Delta_1$).
\end{enumerate}
 We proceed by induction and assume that $c^*_\gamma$ has been defined for all $\gamma\kin\Gamma_n=\bigcup_{j=1}^n \Delta_n$.
 Assume that $\gamma\in\Delta_{n+1}$ with $n+1\not\in\{m_j:j\in\N\}$. Let $\gamma=(r_1x^*_1,r_2x^*_2, \ldots,r_\ell x^*_\ell)$. There are several cases.
 \begin{enumerate}
 \item[ii)]   $\ell=1$, so $\gamma=(r_1x^*_1)$,  where $|\supp_{\bE^*}(x_1^*)|>1$. Let
 $(s_1y^*_1,s_2 y^*_2,\ldots,s_my^*_m)$ be the special $c$-decomposition of $x^*_1$ and note that $m\ge 2$, since $\|x^*_1\|\ge \sfrac12>c$.
  Put $\xi=(s_1y^*_1,s_2y^*_2,\ldots, s_{m-1}y^*_{m-1})$ and let $\eta$ be the special $c$-decomposition of $y^*_m$.
Define $c^*_\gamma=r_1e^*_\xi+r_1s_me^*_\eta$.
 \item[iii)] $\ell=2$ and $|\supp_{\bE^*}(x^*_1)|=1$. Let $\xi=(x_1^*)$ and let $\eta$ be the special $c$-decomposition of $x^*_2$ and set
 $c^*_\gamma=r_1e^*_\xi+r_2 e^*_\eta$.

 \item[iv)] $\ell>2$ or $\ell=2$ and $|\supp_{\bE^*}(x_1^*)|>1$. Let $\xi=(r_1x^*_1,r_2x^*_2, \ldots r_{\ell-1} x^*_{\ell-1})$ and let $\eta$ be the special $c$-decomposition of $x^*_\ell$.
  Define $c^*_\gamma=e^*_\xi+r_\ell e^*_\eta$.
 \end{enumerate}

 Note that in the  cases (ii), (iii) and (iv)
  $k:=\rk(\xi)<\rk(\eta)\le n$ and, furthermore, as can be shown inductively
\begin{equation}\label{E:4.9}
\min\supp_{\bF^*}  (e^*_\gamma)  \ge m_{\min\ran_{\bE^*}(\gamma)} \text{ for all $\gamma\in\Delta_n$}.
\end{equation}
For the recoding we proceed as follows. We will identify $\Delta_n$
with new sets $\tilde \Delta$  conforming to Definition \ref{D:1.1}.
Set $\tilde\Delta_1=\Delta_1=\{(rx^*):r\in R_1, x^*\in A^*_1\}$. For
$n\ge 2$ we
 will identify $\Delta_n$ with  $\tilde\Delta_n=\tilde\Delta_j^{(0)}\cup\tilde\Delta_j^{(1)}$.
Assume this has be done for $j\le n$. We let $\gamma\in\Delta_{n+1}$ and define $\tilde\gamma$ in the four cases  above.
\begin{enumerate}
\item[i)] If $\gamma=(rx^*)$ with $r\in R_j$ and $x^*\in A^*_j$ for some $j\in\N$, and thus $\rk(\gamma)=m_j$, we let
 $\tilde\gamma=(m_j,0,0,rx^*)$, i.e. we choose $\beta=0$, $b^*=0$ and $(rx^*)$ to be the free variable.
 \end{enumerate}
 In the next three cases let $\xi$, $\eta$ and $k=\rk(\xi)$, $\ell$,$ m$, $r_j$, $j\le\ell$, and $s_j$, $j\le m$, be as above in (ii), (iii) and (iv), and let $\tilde \xi$ and $\tilde \eta$ be the recodings of $\xi$ and $\eta$.
 \begin{enumerate}
 \item[ii)] If $\gamma=(r_1x^*_1)$, with $|\supp_{\bE^*}|>1$, we let $\tilde\gamma=(n+1,2r_1,\frac12(e^*_{\tilde\xi}+s_me^*_{\tilde\eta}))$.
 \item[iii)]   If $\gamma=(r_1x^*_1,r_2x^*_2)$, with $|\supp_{\bE^*}(x_1^*)|=1$,  let  $\tilde\gamma=(n+1,r_1,k,\tilde\xi, r_2,e^*_{\tilde\eta})$.
 \item[iv)]   If $\gamma=(r_1x^*_1,r_2x^*_2, \ldots, r_\ell x^*_\ell)$,  with $\ell>2$  or $|\supp_{\bE^*}(x_1^*)|>1$,
  let  $\tilde\gamma=(n+1,1,k,\tilde\xi, r_\ell,e^*_{\tilde\eta})$.
\end{enumerate}
In cases (i) and (ii), $\tilde\gamma$ is of type 0, while in the
other cases it is of type 1. In cases (ii),(iii) and (iv) the set of
free variables is a singleton and we have thus suppressed it.
Definition \ref{D:1.3} yields that the Bourgain-Delbaen space
corresponding to the $\tilde\Delta_n$'s
 is exactly the same as the one obtained from  the $\Delta_n$'s above. Indeed, in (ii), (iii) and (iv) the definition
  of $c^*_{\tilde\gamma}$ involves  the projections $P^{\bF^*}_{(k,n]}$. But
 $P^{\bF^*}_{(k,n]}(e^*_\eta)=e^*_\eta$ by Proposition \ref{P:partialorder} and \ref{E:4.9}.
Also, from our construction, we note that \eqref{E:1.1a}  is
satisfied for  the $\tilde\Delta_n$'s since the  factors $r$
involved are all at most $2c\le \sfrac18$, unless
 the relevant  $b^*=e^*_\eta$ and $c^*_\eta=0$, for some $\eta\kin\Gamma$. It follows as in  Remark  \ref{R:1.4a}, that $\bF^*=(F_j^*)$ is an FDD for $\ell_1$,
  whose decomposition constant $M$ does not exceed $2$.

Let $\gamma=(r_1x^*_1,\ldots, r_\ell x^*_\ell)\in\Gamma$, $\ell\ge
2$. Then by iterating case (iv) we can compute the analysis of
$e^*_\gamma$. Namely
 $e^*_\gamma=\sum_{j=3}^\ell (d^*_{\gamma_j}+ r_j e^*_{\eta_j}) +e^*_{\gamma_2}$,
where $\gamma_j=(r_1x^*_1,\ldots,r_\ell x^*_\ell)$, for $2\le j\le
\ell$, and $\eta_j$ is the special $c$-decomposition of $x^*_j$, for
$3\le j\le \ell$. By considering the different cases where
$|\supp_{\bE^*}(x_1^*)|$  has one or more elements we have

\begin{equation}\label{E:4.10} e^*_\gamma=
\begin{cases} \sum_{j=1}^\ell d^*_{\gamma_j}+ r_j e^*_{\eta_j} &\text{if $|\supp_{\bE^*}(x_1^*)|=1$}\\
 \sum_{j=2}^\ell (d^*_{\gamma_j}+ r_j e^*_{\eta_j} )+d^*_{\gamma_1} + r_1 e^*_{\xi'}+r_1s_m e^*_{\eta'}        &\text{if $|\supp_{\bE^*}(x_1^*)|>1$,}
 \end{cases}
\end{equation}
where in the bottom displayed formula, using case (ii),
$\xi'_1=(s_1y^*_1,\ldots,s_{m-1}y^*_{m-1})$, where
  $(s_1y^*_1,\ldots,s_{m-1}y^*_{m-1}, s_m y^*_m )$
is the  special $c$-decomposition  of $x^*_1$) and $\eta'$ is the
special $c$-decomposition of $y^*_m$.

From \ref{E:4.10}, Corollary  \ref{C:3.7} and our construction using special $c$-decompositions of elements of $D$, it follows that $(F_i)$ is a shrinking FDD, if $(E_i)$ is a shrinking FDD.
Indeed, then the set $\{(\min\supp_{\bE^*} x^*_i)_{i=1}^\ell : (r_1x^*_1,\ldots ,r_\ell x^*_\ell)\kin\Gamma\}$ is compact. From the analysis \eqref{E:4.10}
 we see that $\cC=\{\cuts(\gamma):\gamma\in\Gamma\}$ is also compact.

 To complete the proof of Theorem~A it remains only to show that $X$ embeds into $Y$, the Bourgain-Delbaen space associated to $(\Delta_n)$.
  As in Section \ref{S:2} we let $J_m:\ell_\infty(\Gamma_m)\to Y\subset \ell_\infty(\Gamma)$ be the extension operator, for $m\in\N$.

\begin{defn}\label{D:4.4}
For $i\!\in\! \N$, define
$\phi_i :E_i \to \ell_\infty (\Delta_{m_i})$
by $\phi_i (x) (rx^*) = rx^* (x)$.
Define
$\phi : c_{00} (\bigoplus_{i=1}^\infty E_i) \to Y = \overline{\bigcup_m Y_m}
\subseteq \ell_\infty (\Gamma)$
by $\phi (x) = \sum_i J_{m_i} \circ \phi_i (P_i^{\bE} x)\!\in\! c_{00}
(\bigoplus_{i=1}^\infty F_{m_i})$.
\end{defn}

In proving that $X$ embeds into $Y$ we will use the following connection between
the functionals $e^*_\gamma$ and the elements $\gamma\in\Gamma$
 deriving from the  elements of $D$.
 \begin{align}\label{E:4.13a}
 &\text{If $n\not\in\{m_j:j\in\N\}$ and $\gamma=(r_1x^*_1,\ldots, r_\ell x^*_\ell)\in\Delta_n$, then $c_\gamma^*=\alpha e^*_\xi+\beta e^*_\eta$,}\\
  &\text{where $\xi=(s_1y_1^*,s_2y_2^*,\ldots, s_k y^*_k)$ and  $\eta=(t_1z^*_1,\ldots, t_mz^*_m)$ are in $\Delta_{n-1}$, such that} \notag \\
  &\qquad \sum_{i=1}^\ell r_i x^*_i=\alpha \sum_{i=1}^\ell s_i y^*_i+\beta\sum_{i=1}^\ell t_i z^*_i.\notag
 \end{align}
This is easily verified using (ii), (iii) and (iv).
 Note that, since $A^*_i\subset B_{E_i^*}$ is $(1-\vp/4)$-norming  $E_i$,
 $(1-\vp/4)\|x\|\le \|\phi_i(x)\|\le \|x\|$ for all $x\in E_i$.
\begin{prop}\label{P:map-phi}
The map $\phi$ extends to a  isomorphism of $X$ into $Y$,
and
 $$(1-\vp)\|x\|\le \|\phi(x)\|\le \|x\| \text{ for all $x\kin X$}.$$
\end{prop}

\begin{proof}
Using \eqref{E:4.13a}  and the definition of $\phi_j$, $j\in\N$, we
deduce, by induction on the rank of $\gamma\in\Gamma$, that for all
$\gamma=(r_1x^*_1,\ldots, r_\ell x^*_\ell)\in\Gamma$ and all $x\in
\coo(\oplus_{j=1}^\infty E_j)$,
$$e^*_\gamma(\phi(x))=\sum_{j=1}^\ell r_j x^*_j(x).$$
Using the bimonotonicity of $\bE$ in $X$, and the properties of the set $D\subset B_{X^*}$ as
 listed in Lemma \ref{L:4.1} we obtain for $x\in\coo(\oplus_{j=1}^\infty E_j)$
 $$(1-\vp)\|x\|\le \sup_{x^*\in D} |x^*(x)|=\sup_{\gamma=(r_1x^*_1,\ldots, r_\ell x^*_\ell)\in \Gamma}
\Big|\sum_{i=1}^\ell r_jx^*_j(x)\Big|=\sup_{\gamma\in \Gamma}\big| e^*_\gamma(\phi(x))\big|\le \|x\|,$$
which implies our claim.\end{proof}

We will be using the construction of $Y$ and all the terminology and
notation of that construction in the next two sections. In the proof
of Theorems~B and C we will also be using the construction for $V$
replacing $X$ where $V$ has a normalized bimonotone basis
$(v_i)_{i=1}^\infty$. In this case the $v_i$'s play the role of the
$E_i$'s, more precisely $E_i$ is replaced by $\text{span}(v_i)$. To
help distinguish things we will write $BD_X$ and $BD_V$ for the
respective $\cL_\infty$ spaces containing isomorphs of $X$ and $V$.

Finally, it is perhaps worth noting that, in the $V$ case we could
alter the proof slightly by allowing the scalars  $R_i$ to be
negative and $\vp_i/8$-dense in $[-1,1]\setminus \{0\}$ and take
$A_j^* = \{\frac1{1+\vp/4}v_j^*\}$. In the case that $(v_i)$ is also
1-unconditional we can use $A_j^* = \{v_j^*\}$ (see the second part
of Lemma \ref{L:4.1}). We would then obtain

\begin{cor}\label{C:alteredproof}
Let $V$ be a Banach space with a normalized bimonotone shrinking
basis $(v_i)_{i=1}^\infty$. Then $W$ embeds into a $\cL_\infty$
space $Z$, with a shrinking basis $(z_i)_{i=1}^\infty$ so that
$(v_i)_{i=1}^\infty$ is equivalent to some subsequence of
$(z_i)_{i=1}^\infty$.
\end{cor}
In  case that $V$ is the   Tsirelson space  $T_{c,\alpha}$   the
construction of a Bourgain-Delbaen space containing $V$ becomes
simpler.

\begin{rem}\label{R:4.5}  Let $X$ be the Tsirelson space $T_{c,\alpha}$,
where $\alpha<\omega_1$ and $c\le \sfrac1{16}$.
 In $T_{c,\alpha}^*$ there is a natural choice  for  the set $D$  satisfying the conditions of  Lemma \ref{L:4.1} ($1$-unconditional case). Indeed, we
  let
 $D=\bigcup_{n=0}^\infty D_n$, where $D_n$, $n\ge 0$ is defined
 by induction
 \begin{align}
 D_0&=\{\pm e^*_j: j\in\N\}\text{ and assuming $D_0,D_1\ldots D_n$ have been defined we let }\\
 D_{n+1}&=\left\{ c\sum_{i=1}^k x^*_i :
  \begin{matrix} k\!\ge\!2,\, x^*_i\in \bigcup_{j=0}^n D_j, \text{ for $i\le k$},\quad \{\min\supp(x^*_i):i\le k\}\in S_\alpha, \\
 \text{and }\max \supp(x^*_i)<\min\supp(x^*_{i+1}), \text{if $i<k.$}
   \end{matrix}
  \right\}.\notag
 \end{align}
 In that case $D$ 1-norms $T_{c,\alpha}$
     and  $\Gamma$ also has a simple form in this case:
  \begin{equation*}
  \Gamma_{\alpha,c}=\left \{(cx^*_1,cx^*_2, \ldots,c x_\ell^*): \begin{matrix} \ell\!\ge\!2, x^*_i\in D,  \text{ for $i\le \ell$,}\, \{\min\supp(x^*_i):i\le \ell\}\in S_\alpha, \\
 \text{and }\max \supp(x^*_i)<\min\supp(x^*_{i+1}), \text{if $i<\ell,$}
   \end{matrix}
  \right\}\cup  D_0.
  \end{equation*}
   Our construction  in  Theorem A  leads  then to a Bourgain-Delbaen space containing
   isometrically $T_{c,\alpha}$  and it is very similar  (but simpler) than the construction in \cite{AH}
   where a {\em mixed Tsirelson space}  was used instead of  $T_{c,\alpha}$.
\end{rem}
In summary, our proof of Theorem A, then yields the following
theorem.

\begin{thm}\label{T:4.9}  Let $X$ be a Banach space with a bimonotone FDD $\bE=(E_j)$ and let $\vp>0$. Then $X$ embeds into a Bourgain-Delbaen space $Z$ having an FDD $\bF=(F_j)$, such that
\begin{enumerate}
\item[a)] For $n\in\N$, there are embeddings
$\phi_n: E_n\to F_{m_n}$, so that
$$\phi:\coo\big(\oplus_{n=1}^\infty E_n\big)\to Z,\quad \sum x_n\mapsto \sum \phi_n(x_n)$$
extends to an isomorphism from $X$ into $Z$ with $(1-\vp)\|x\|\le
\|\phi(x)\|\le \|x\|$ for $x\in X$.
\item[b)] $\bF$ is shrinking (in $Z$) if $\bE$ is shrinking (in $X$).
\end{enumerate}
\end{thm}
From Theorem \ref{T:4.9} and \cite[Corollary 3.5]{FOSZ} we obtain

\begin{cor}\label{C:4.10} There exists a collection $\{Y_\alpha:\alpha<\omega_1\}$ of
$\cL_{\infty,2}$ spaces such that
 $Y^*_\alpha$ is $2$-isomorphic to $\ell_1$, and
 $Y_\alpha$ is universal for the class
 $\mathcal D_\alpha=\big\{X: X\text{ separable and }S_z(X)\le \alpha\big\},$
 for all $\alpha<\omega_1$.
\end{cor}

\section{The proof of Theorems B and C}\label{S:5}

 The constructions which will be used to prove Theorems B and C  are {\em augmentations} of  sequences of  Bourgain-Delbaen sets as introduced in
 Section \ref{S:2}.
\begin{defn}\label{D:augmentation}
Assume that $(\Delta_n)$  is a sequence of Bourgain-Delbaen  sets,
and assume
  that $(\Delta_n)$ satisfies the assumptions of Proposition $\ref{P:1.4}$ with $C<\infty$, and hence $M<\infty$. We denote the Bourgain-Delbaen space associated with
     $(\Delta_n)$ by $Y$ and its FDD by $\bF=(F_n)$.
      Since we will deal with different Bourgain-Delbaen spaces
     we denote from now on the projections $P_A$ of $Y$ onto $\oplus_{j\in A} F_j$, $A\subset \N$ finite or cofinite,
     by $P^{\bF}_A$.

 An {\em augmentation of $(\Delta_n)$}, is then a sequence of finite, possibly empty,  sets $(\Theta_n)$ having the property
 that $(\Deltab_n):=(\Delta_n\cup \Theta_n)$ is again a sequence of Bourgain-Delbaen sets. More concretely, this means the following.
 $\Theta_1$ is a finite set and assuming that for some $n\in\N$, $(\Theta_j)_{j=1}^n$  have been chosen, we let
 $\Deltab_j=\Delta_j\cup \Theta_j$,  $\Lambda_j=\bigcup_{i=1}^j \Theta_i$, and $\Gammab_j=\bigcup_{i=1}^j \Deltab_i$, for $j\le n$, where
 $\Theta_{n+1}$ is the union of two sets, $\Theta_{n+1}^{(0)}$ and $\Theta_{n+1}^{(1)}$, which
  satisfy the following conditions.

 $\Theta_{n+1}^{(0)}$ is  finite and
 \begin{align}\label{E:6.2.1}
 \Theta_{n+1}^{(0)}\subset \big\{(n+1,\beta,b^*,f): \beta\kin[0,1], b^*\kin B_{\ell_1(\Gammab_n)}, \text{ and }f\kin W_{(n+1,\beta,b^*)}\big\},
 \end{align}
 where $W_{(n+1,\beta,b^*)}$ is a finite set for $\beta\kin[0,1]$ and $b^*\kin B_{\ell_1(\Gammab_n)}$.

$\Theta^{(1)}_{n+1}$ is finite and
 \begin{equation}\label{E:6.2.2}
\Theta_{n+1}^{(1)} \subset \left\{ (n+1, \alpha, k,\xib, \beta, b^*,f):
 \begin{matrix} \alpha,\beta\kin[0,1], k\kin\{1,2,\ldots n-1\}, \xib\kin \Deltab_k, \\
b^*\kin B_{\ell_1(\Gammab_n\setminus \Gammab_k)} \text{ and } f\kin
W_{(n+1,\alpha, k,\xib,\beta,b^*) }\end{matrix}
\right\},\end{equation} where $W_{(n+1,\alpha, k,\xib,\beta,b^*)}$ is
a finite set for $\alpha\kin[0,1]$, $k\kin\{1,2,\ldots, n-1\}$,
$\xib\kin\Deltab_k$,  $\beta \kin[0,1]$, and  $b^*\kin
B_{\ell_1(\Gammab_n\setminus\Gammab_k)}$.

We denote the
 corresponding functionals (see Definition \ref{D:1.3}) by $c_{\gammab}^*$ for $\gammab\in\Gammab$.   We require also that
  $(\Deltab_n)$ satisfies  the conditions of Proposition \ref{P:1.4},
 so that $\bFb^*=(\Fb_n^*)$, with $\Fb_n^*=\spa(e^*_{\gammab}:\gammab\in \Deltab_n)$ is an FDD of $\ell_1(\Gammab)$ whose decomposition constant
   $\Mb$   can be estimated as in Proposition \ref{P:1.4}.  We denote then the associated Bourgain-Delbaen space by $Z$, and its FDD by $\bFb=(\Fb_n)$.
   As in Section \ref{S:2}, we denote the projections from $Z$ onto $\oplus_{i=k}^m \Fb_i$, by $P^{\bFb}_{[k,m]}$, if $k<m$,  or by  $P^{\bFb}_k$,  if $k=m$.
   The restriction operator from $\ell_\infty(\Gammab)$ onto $\ell_\infty(\Gammab_n)$ or   $\ell_1(\Gammab)$ onto $\ell_1(\Gammab_n)$ is denoted by
   $\Rb_n$ and the extension operator from $\ell_\infty(\Gammab_n)$ to $\oplus_{j=1}^m \Fb_j\subset Z\subset \ell_\infty(\Gammab)$ is denoted by $\Jb_m$.

   Note that by Corollary \ref{C:3.7}, under assumption \eqref{E:1.1a},  $\bFb$  is shrinking in $Z$ if  $\{\textrm{cuts}(\gamma):\gamma\in\Gamma\}$ is compact.

  \end{defn}

 \begin{rem}\label{R:5.2}
   In general $Y$ is not a subspace of $Z$. Nevertheless it follows from Proposition \ref{P:1.4a}  that
   $F_m$ is  naturally isometrically embedded into $\Fb_m$
   for $m\kin\N$. Indeed, the map
   $$\psi_m: F_m\to \Fb_m,  \quad x\mapsto \Jb_m J_m^{-1} (x)=\Jb_m(x|_{\Delta_m}),$$
   is an isometric embedding (where we consider $\ell_\infty(\Delta_m)$  to be naturally embedded  into $\ell_\infty(\Deltab_m)$ and   $\ell_\infty(\Deltab_m)$ naturally
   embedded into  $\ell_\infty(\Gammab_m)$).
   We put
   \begin{equation}\label{E:5.2}
   \psi:\coo\big(\oplus_{j=1}^\infty F_j\big)\mapsto \coo\big(\oplus_{j=1}^\infty \Fb_j\big),
    \quad (x_j)\mapsto (\psi_j(x_j)).   \end{equation}
    We define $\psi$ on $(\oplus_{j=1}^\infty F_j)_{\ell_\infty}$ by $\psi((x_j)_{j=1}^\infty)=(\psi_j(x_j))_{j=1}^\infty\in\prod_{j=1}^\infty
    \Fb_j$, a sequence in $(\Fb_j)_{j=1}^\infty$.  Note that
    if $\gammab\in\Lambda_n$ then we can regard, for $x=(x_j)\in(\oplus
    F_j)_{\ell_\infty}$, $c^*_\gammab(\psi(x))=c^*_\gamma(\sum_{j=1}^n
    \psi_j(x_j))$.  It is worth noting that for $y\in\coo\big(\oplus_{j=1}^\infty F_j\big)$,
    $\psi(y)|_\Gamma=y$. Thus $\psi$ extends such elements to elements of $Z$. However this extension  is not necessarily bounded on $Y$. In any event, if we define $\pi(z)=z|_\Gamma$ for
    $z\in Z$ then  $\pi:Z\to Y$.
\end{rem}
 The following provides a sufficient criterium for a subspace of $Y$ to also embed into the augmented space  $Z$.

  \begin{prop}\label{P:5.2}   Assume that $X$ is a subspace of the Bourgain-Delbaen space $Y$
with FDD $\bF=(F_j)$ and   which  is associated to a
Bourgain-Delbaen sequence $(\Delta_n)$.
  Assume moreover that $\coo(\oplus_{j=1}^\infty F_j)\cap X$ is dense in $X$.

 Let $(\Theta_n)$
 be an augmentation of $(\Delta_n)$ with an associated space $Z$,
 and assume that    $|c_{\gammab}^*(\psi(x))|\le c_X\|x\|$ for all  $\gammab\in\Lambda=\bigcup_{j\in\N} \Lambda_j$ and all $x\in X$.
   Then $\psi$ embeds $X$ into $Z$ and $\|x\|\le \|\psi(x)\|\le
   \max(1,c_X)\|x\|$.  Furthermore, for $x\in X$, $\pi(\psi(x))=x$.
   Thus $\pi:\psi(X)\rightarrow X$ is the inverse isomorphism of
   $\psi|_X$.
  \end{prop}

  \begin{rem}\label{R:5.3} In \cite[Lemma 3.1]{OS} and \cite{JRZ} it was shown that every separable Banach space
  $X$ can be embedded into  a  Banach space  $W$ with FDD $\bE=(E_j)$,
   so that $X\cap \coo(\oplus_{j=1}^\infty E_j)$ is dense in $X$.  Moreover, $(E_j)$ can be chosen to be shrinking if $X^*$ is separable. Using the construction of Theorem A, we can therefore  embed $W$ into a Bourgain-Delbaen space $Y$ which has an FDD $\bF=(F_j)$
   so that $E_j$ embeds into $F_{m_j}$ for  some increasing sequence $(m_j)$.
   It follows therefore that the image of $X$ under the embedding into $Y$ has the property
   needed in Proposition \ref{P:5.2}.
  \end{rem}

  \begin{proof}[Proof of Proposition \ref{P:5.2}]
   For $x\kin X$ and $\gammab\kin\Gammab$ we first estimate $e^*_{\gammab}(\psi(x))$.
 If $\gamma\kin \Gamma$ then
  $e^*_\gamma(\psi(x))\! =\!e^*_\gamma(x)$, and thus it follows that $\|\psi(x)\|\!\ge\! \|x\|_{\ell_\infty(\Gamma)}=\|x\|$ for all $x\kin X$ and $\pi(\psi(x))=x$. If $\gammab\kin\Lambda$ it follows that
  $$|e^*_{\gammab}(\psi(x))|=|c^*_{\gammab}(\psi(x))|\le c_X\|x\|$$
   and therefore the restriction of  $\psi$ to $X$ is a bounded operator, still denoted by $\psi$, from $X$ to $\ell_\infty(\Gammab)$,
   and  $\|\psi\|\le \max(c_X, 1) $.

   We still need to show that the image of $X$ under $\psi$ is contained in $Z$.
However $\psi(X\cap\coo(\oplus_{j=1}^\infty F_j))\subset Z$ since
$\psi(X\cap F_j)\subset\psi(F_j)\subset \Fb_j\subset Z$ for all
$j\in\N$.  Thus the image of $\psi$ on a dense subspace of $X$ is
contained in $Z$, and hence $\psi(X)\subset Z$.
  \end{proof}

     \begin{thm}\label{T:6.1}
  Let $Y$ be the Bourgain-Delbaen space associated to  a sequence of sets $(\Delta_n)$ and let  $\bF=(F_n)$ be the FDD of  $Y$.
   Let $X$ be a subspace of $Y$ and assume that $\coo(\oplus_{j=1}^\infty F_j)\cap X$ is dense in $X$ and
 let  $V$ be a space with a 1-unconditional,  and normalized basis
  $(v_n)$.

  Then there is an augmentation $(\Theta_n)$ of $(\Delta_n)$ with an associated
  space $Z$ and with   FDD  $\bFb=(\Fb_n)$ so that the following hold.
  \begin{enumerate}

  \item[a)] $X$ embeds isometrically into $Z$ via $\psi$.
\item[b)]
 If $\bf{F}$ and $(v_i)$ are shrinking, then $\bFb$ is also shrinking  and, thus,  $Z^*$ is
    isomorphic to $\ell_1$. Furthermore,
  if $(z_n)$ is a  normalized block basis in $Z$, with the property
  that
  $$\delta_0=\inf_{n\in N} \dist(z_n,\psi(X))>0$$
  then $(z_n)$ has a subsequence $(z'_n)$ which dominates $(v_{k_n})$
   where $k_n\!=\!\max\supp_{\bFb}(z'_{n})\!+\!1$, for $n\in\N$.
 \item[c)] If $X$ has an FDD $\bE=(E_n)$, with the property that $E_n\subset F_n$, for $n\in\N$,
  then in this case  we can choose    $(\Theta_n)$ so that
 $$c^*_{\gammab}(\psi(x))=0,
 \text{ whenever, $\gammab\in\Lambda=\bigcup_{j=1}^\infty \Theta_j$ and $x\kin X$}.$$
  Moreover
    every normalized block sequence $(z_n)$ satisfying
    \begin{align}\label{E:6.1.1a}
   \max\supp_{\bFb}(z_{n})+n+2< \min\supp_{\bFb}(z_{n+1}) \text{ and }
       \delta_0=\inf_{n\in N} \dist(z_n,\psi(X))>0,
   \end{align}
   dominates $(v_{k_n})$, where $k_n=\max\supp_{\bFb}(z_{n})+1$.
     \end{enumerate} \end{thm}
     \begin{rem} In case (c)  we allow some $E_n$ to be the nullspace $\{0\}$.
      As noted in the introduction, this will be convenient. In the case of Theorem A, we actually
       had $E_j\subset F_{m_j}$, but we choose to simplify the notation  in the arguments below.
     \end{rem}
   \begin{proof}[Proof of Theorem \ref{T:6.1}]
      The construction of $(\Theta_n)$ will differ slightly  depending on whether $X$ has an FDD or not.

      We   use the construction   of Section \ref{S:4} for the space
      $V$ with $c\leq\sfrac{1}{16}$
using as an FDD for $V$ the basis $(v_i)_{i=1}^\infty$ and
$A^*_j=\{\pm v^*_j\}$ for all $j\in\N$. We write $D^V,
\Delta_n^V,\Gamma_n^V,\ldots$ to distinguish these sets from
$\Delta_n,\Gamma_n,\ldots$ which came from the
 construction of $Y$.
Thus we obtain a $\cL_\infty$ space $Y^V$ and a
$\frac1{1-\vp}$-embedding (see Proposition \ref{P:map-phi}) $\phi^V
:V\to Y^V$.The numbers
 $\vp<c$ and $(\vp_n)\subset(0,c)$ satisfy, as in Section \ref{S:4}, the  condition
 \eqref{eq:4.1}.

Now $D^V=D$ is as defined in the unconditional case of  Lemma
\ref{L:4.1} for the space $V$.
 We also note  that in the case that $V$ is the Tsirelson space, $T_{c,\alpha}$ with $\alpha<\omega_1$ and
 $c\le \sfrac1{16}$  we could  use  $D^V$ and  $\Gamma^V=\Gamma_{c,\alpha}$ as defined in
 Remark \ref{R:4.5}.

 We define by induction for all $n\in\N$  the sets  $\Theta_n$  and  the sets $\Theta_n^{(0)}$ and  $\Theta_n^{(1)}$, if $n\ge 2$, satisfying \eqref{E:6.2.1} and \eqref{E:6.2.2}. Moreover,
 we  also define  a map $\Theta_n\to\Gamma^V$, $\gammab\mapsto \gammab^V$ so that
 \begin{align}\label{E:6.1.1}
 &\cuts(\gammab) \text{ is a spread of $\{\min\supp_{V^*}(x^*_1),\min\supp_{V^*}(x^*_2),\ldots,\min\supp_{V^*}(x^*_\ell)\}$,}\\
 &\text{where $\gammab^V=(x^*_1,x^*_2,\ldots, x^*_\ell)\kin\Gamma^V$,
 for $\gammab\in\Theta_n$, and $\max\supp_{V^*}(\gammab^V)\leq n$.}\notag
 \end{align}
 The set of free variables will be a singleton,  and $\alpha$  will always  be chosen to be $1$ in (\ref{E:6.2.2}), so we suppress the free variable and $\alpha$, in the definition
 of the elements of $\Theta_n$.

 To start the recursive construction we put $\Theta_1=\emptyset$, and assuming $\Theta_j^{(0)}$ and $\Theta_j^{(1)}$
 have been chosen for all $j\leq n$, we proceed as follows.
  $\Lambda_j$, and $\Gammab_j$, $j\le n$, $\Fb^*_j$ and $P^{\bFb^*}_{(k,j]}$, $0\le k<j\le n$, are given as  in Definition \ref{D:augmentation}.
  Since $Y$ is a subspace of $\ell_\infty(\Gamma)$,  and since $\Gamma_n\subset \Gammab_n$,
      $e^*_\gammab$, $ \gammab\in \Gammab_n$, is a well defined functional on $Y$ (and thus on $X$).
The map $\psi:X\rightarrow\prod_{j=1}^\infty \Fb_j$ will be
defined ultimately as in (\ref{E:5.2}).  At this point for $x\in X$,
$\psi(x)|_{\Gammab_n}$ is defined and so
$e^*_\gammab(\psi(x))=c^*_\gammab(\psi(x))$ is defined for
$\gammab\in\Gammab_n$.
  Thus we can choose for $0\leq k<n$, finite  sets
 $$B_{(k,n]}\subset
 \begin{cases}  \{ b^*\in B_{\ell_1(\Gammab_n\setminus\Gammab_k)}: P^{\bFb*}_{(k,n]}(b^*)|_{\psi(X)}\equiv 0 \} &\text{, assuming $X$ has an FDD}\\
                  B_{\ell_1(\Gammab_n\setminus\Gammab_k)}  &\text{, no assumptions on }X \end{cases} $$
which are  symmetric and $\vp_{n+1}/(2M+4)$ dense in their
respective supersets. Then we put
\begin{align*}
\Theta_{n+1}^{(0)}=&\,\Theta_{n+1}^{(0,1)}\cup \Theta_{n+1}^{(0,2)} \text{ with }\\
\Theta_{n+1}^{(0,1)}&=\{(n+1,rc,b^*): ( rv_{n+1}^*)\in\Gamma^V
\text{ and }
 b^*\kin B_{(0,n]}  \}\\
 \Theta_{n+1}^{(0,2)}&=\left\{(n+1,r,e^*_\etab): \begin{matrix} &\etab\in\Lambda_n, \exists
x^*\kin D^V\text{ so that }  (rx^*)\in\Gamma^V\text{ with }|\supp_{V^*}(x^*)|>1 \\
   & \text{ and }\etab^V\text{ is the special c-decomposition of }x^*  \end{matrix}\right\},
 \end{align*}
 and
 \begin{align*}
 \Theta_{n+1}^{(1)}=& \Theta_{n+1}^{(1,1)}\cup \Theta_{n+1}^{(1,2)} \text{ with }\\
 \Theta_{n+1}^{(1,1)}&=\left\{\gammab=(n+1,k,\xib,rc, b^*):  \begin{matrix}    k<n,\, \xib\kin \Theta_k,   b^*\kin B_{(k,n]}, (\xib^V,r v^*_{n+1})\in \Gamma_{n+1}^V,\\
        \text{ with } |c^*_\gammab(\psi(x))|\le \|x\|\text{ for all }x\in X  \end{matrix}  \right\}    \\
 \Theta_{n+1}^{(1,2)}&=\left\{\gammab=(n+1,k,\xib,r, e^*_{\etab}):    \begin{matrix} k\!<\!n,\, \xib\kin \Theta_k,    \eta\kin\Lambda_n, \exists x^*\kin D^V\text{ with }|\supp(x^*)|\!>\!1, \text{ so}\\
   \text{that } (\xib^V,rx^*)\kin\Gamma_{n+1}^V,\text{ and } \etab^V\text{ is the special $c$-decom-}\\
   \text{position of $x^*$ with } |c^*_\gammab(\psi(x))|\le \|x\|\text{ for all }x\in X      \end{matrix}\right\}.
 \end{align*}
Note that for $(n+1,r,e^*_{\etab})\in\Theta^{(0,2)}_{n+1}$ or
$(n+1,k,\xib,r,e_\etab^*)\in\Theta^{(1,2)}_{n+1}$ we have that
$r\leq c$ since $|\supp(x^*)|>1$.
 We define for $\gammab\in\Lambda_n$, $n\ge 2$,
 $$\gammab^V=\begin{cases}(rv^*_{n+1}) &\text{\!\!if $\gammab=(n+1,rc,b^*)\in \Theta_{n+1}^{(0,1)}$,}\\
                                           (rx^*) &\text{\!\!if  $\gammab=(n+1,r,e^*_\etab)\kin \Theta_{n+1}^{(0,2)}$,}
                                                                                                                  \text{ where $\etab^V$ is the  special c-decomposition  of $x^*$,}\\
                                                     (\xib^V,rv^*_{n+1})     &\text{\!\!if  $\gammab=(n+1,k,\xib,rc, b^*)\in \Theta_{n+1}^{(1,1)}$,}\\
                                                   (\xib^V,rx^*)   &\text{\!\!if  $\gammab\!=\!(n+1,k,\xib,r, e^*_\eta)\kin\Theta_{n+1}^{(1,1)}$,}
                                                                        \text{where $\etab^V$ is the special c-decomposition of $x^*$.}
                         \end{cases}$$

        Then  condition (\ref{E:6.1.1}) follows immediately for the elements of $\Theta_{n+1}^{(0)}$, while an easy induction argument proves it also for the elements of
       $\Theta_{n+1}^{(1)}$.  It is worth pointing out that $\{\gammab^V:\gammab\in\Lambda\}$ is a proper subset of $\Gamma^V$, but nevertheless is sufficiently large for our purposes.

       Proposition \ref{P:1.4}   yields that  $(\Deltab_n)$ admits an associated Bourgain-Delbaen space $Z$ with FDD $\bFb=(\Fb_j)$
       whose decomposition\ constant $\Mb$ is  not larger than  $\max(M, 1/(1-2c))\le \max (M,2)$, where $M$ is the decomposition constant of $(F_j)$.
       If $(F_j)$ and
       $(v_n)$ are both shrinking in $V$, and thus, the optimal $c$-decompositions of elements of $B_{V^*}$  are admissible with respect to some
        compact subset of $[\N]^{<\omega}$, our
                  condition (\ref{E:6.1.1}) together with Theorem \ref{P:bimonotone} and Corollary \ref{C:3.7}  yield that the FDD $\bFb=(\bFb)$  is shrinking in
        $Z$.  The definition of  $\Theta_n^{(1)}$  together with Proposition \ref{P:5.2} imply that $\psi$ isomorphically embeds $X$ into $Z$.

        To verify parts (b) and (c) of our Theorem and will need the following
\renewcommand{\qedsymbol}{}
             \end{proof}

         \begin{lem}\label{L:6.3}
       Let $(z^*_j)$ be a   block basis in $Z^*$ with respect to $\bFb^*$ and  $(\delta_j)\subset [0,1]$ with $\sum_{j\in \N} \delta_j\le 1$.
       Assume that $|z^*_j(\psi(x))|\le \delta_j$ for all $j\kin\N$ and $x\in B_X$.
Define for $n\in\N$  $p_n=\min\supp_{\bFb^*}(z^*_n) -1$ and
$q_n=\max\supp_{\bFb^*}(z^*_n)+1$ (thus
$\supp_{\bFb^*}(z^*_n)\subset(p_n,q_n)$) and assume  that
        \begin{align}
        \label{E:6.3.2} &z^*_n=P^{\bFb*}_{(p_n,q_n)}(\tilde z_n^*)\text{
for some $\tilde z_n^*\in B_{(q_n,p_n)}$,}\text{ and }
           q_n+n< p_{n+1}.
         \end{align}
           Then for any  sequence  $(\beta_j)_{j=1}^N$ with $w^*=
\sum_{j=1}^N\beta_j v^*_{q_j} \in D^V$ there exists
        $\gammab\in \Lambda_{N+ q_N }$  so that
       \begin{equation}  \label{E:6.3.0a} P^{\bFb^*}_{(p_n,q_n)}(e^*_\gammab) =c \beta_n z^*_n,  \text{ for all $n\!\le\!N$, and } P^{\bFb^*}\!\!(e^*_\gammab)(\psi(x))= \sum_{n=1}^N c \beta_n z^*_n(\psi(x))
       \text{ if $x\in X$}  . \end{equation}
              \end{lem}

       \begin{proof}
       We prove  our claim by induction on $N\in\N$.  If
$N=1$ then $w^*=\pm  v^*_{q_1}$, and we
        let $\gammab=(q_n,c, \pm\tilde z^*_{1}) \in
\Theta^{(0,1)}_{q_1}$. Then
       $e^*_\gammab=d^*_\gammab\pm c \tilde z^*_1$ and
$P^{\bFb^*}_{(p_1,q_1)} (e^*_\gammab)=\pm c z^*_1$, depending on
whether
       $\beta_1=\pm1$. Since  $d^*_\gammab(\psi(x))=0$ for $x\in X$ we also deduce the second part of \eqref{E:6.3.0a}.

       Assume that our claim holds true for  $N$ and let
       $w^*= \sum_{j=1}^{N+1}\beta_j v^*_{q_j} \in D^V$. Then, by our choice of
$D^V$ (see Lemma \ref{L:4.1}),
       $w^*$ has a special $c$-decomposition $(r_1w^*_1, \ldots, r_\ell
w^*_\ell)$, and we
      write  $w^*_{j}$ as
$w_{j}^*=\sum_{i=N_{j-1}+1}^{N_j}{\beta^{(j)}_i}v_{q_i}^*$
       with $\beta_i^{(j)}=\beta_i/r_j$,  for $j\le \ell$ and $N_{j-1}+1\le
i\le N_j$ and $N_0=0<N_1<\ldots N_\ell=N+1$.
  Since  $\ell\ge 2$, we can apply
       the induction hypothesis to each  $w^*_j$  and obtain $\etab_j\in
\Lambda_{q_{_{_{N_j}}}+ N_j-N_{j-1}}$, $j=1,2\ldots \ell$,
       so that $P^{\bFb^*}_{(p_n,q_n)}(e^*_{\etab_j}) = c\beta_n^{(j)} z^*_n$
if  $N_{j-1}< n\le N_j$.
       Now let
       $$\gammab_1=\begin{cases}(q_{1},cr_1,\sign(\beta_1)\tilde z^*_1)\} &\text{if  $|\supp(w^*_1)|=1$}\\
       (p_{_{N_1+1}},r_1,e^*_{\etab_1})                                                                     &\text{if $|\supp(w^*_1)|>1$}. \end{cases}$$
         Note that,  in the second   case, by assumption
\eqref{E:6.3.2}   $q_{_{N_1}}+N_1<  p_{N_1+1}$ and thus $\etab_1\in
\Lambda_{p_{_{N_1+1}}-1}$.
         Assuming we have chosen $\gammab_{j-1}$, for $2\le j\le\ell$ we
let
         $$\gammab_j=\begin{cases}(q_{_{N_j}},\gammab_{j-1},cr_j,\sign(\beta_{_{N_j}})
\tilde z^*_{_{N_j}})&\text{if  $|\supp(w^*_1)|=1$}\\
             (q_{_{N_j}}+
N_j-N_{j-1}+1,\gammab_{j-1},\rk(\gamma_{j-1}),r_j,e^*_{\etab_j})
&\text{if $|\supp(w^*_1)|>1$.}
\end{cases}$$
Using the induction hypothesis on the $\etab_j$'s, we deduce by
induction on   $j=1,\ldots \ell$ that for $x\in X$
$$e^*_{\gammab_j}(\psi(x))=c^*_{\gammab_j}(\psi(x))\le  \sum_{n=1}^{N_j} |c\beta_n z^*_n(\psi(x))|\le \sum_{n=1}^{N_j} \delta_n\|x\|\le\|x\|,$$
and thus $\gammab_1\in \Theta_{q_1}^{(0,1)}$, if $|\supp(w^*_1)|=1$,
and $\gammab_1\in \Theta_{p_{_{N_1+1}}}^{(0,2)}$, if
$|\supp(w^*_1)|>1$, and
  $\gammab_j\in \Theta_{q_{_{N_j}}}^{(1,1)}$, if  $|\supp(w^*_1)|=1$,  and $\gammab_j\in \Theta_{q_{_{N_j}}+N_j-N_{j-1}+1}^{(1,2)}$, if  $|\supp(w^*_1)|>1$, if $j=2,3\ldots \ell$

        Finally we choose $\gammab=\gammab_\ell$ which in both cases  is an element of $ \Lambda_{q_{N+1}+ N+1}$.
                         It follows for $n\le N$, and $1\le j\le\ell$
such that $N_{j-1}<n\le N_j$ that
                 $$P^{\bFb^*}_{(p_n,q_n)}(e^*_\gammab)=P^{\bFb^*}_{(p_n,q_n)}(e^*_{\gammab_j})=\left.\begin{cases}
c r_j \sign(\beta_j)    z^*_n
&\text{ if $|\supp(w^*_j)|=1$}\\

r_jP^{\bFb^*}_{(p_n,q_n)}(e^*_{\etab_j})      &\text{ if
$|\supp(w^*_j)|>1$}   \end{cases} \right\}   =\beta_n c z^*_n,$$
       which finishes  the verification of the first part of \eqref{E:6.3.0a}, while the second part follows from the induction hypothesis applied to the $\etab_j$'s.       \end{proof}

          \begin{proof}[Continuation of  the Proof of Theorem \ref{T:6.1}] To finish the proof
           we consider a normalized block basis  $(z_n)$ in $Z$, with
           $\delta_0=\inf_n \text{dist}(z_n, \psi(X))>0 $ and the additional property
           \eqref{E:6.1.1a} in the case where $X$ has an FDD.
           Let $p_n=\min\supp_{\bFb}(z_n)-1$ and $q_n=\max\supp_{\bFb}(z_n)+1$.
           It follows  that $q_n+n< p_{n+1}$, for $n\in\N$.
   In this case ($X$ has an FDD) we choose $z^*_n\in\oplus_{j\in(p_n,q_n)} \Fb^*_j$, with
   $\|z^*_n\|\le1$, $      z^*_n(z_n)\ge \frac{\delta_0}{2\Mb}$ and $z_n^*|_{\psi(X)}=0$.

     In the case (b)    we proceed as follows.
     We choose $y^*_n\in Z^*$, $\|y^*_n\|\le1$, so that $y^*_n(z_n)\ge \delta_0$ and $y^*_n|_{\psi(X)}\equiv 0$.
           After passing to subsequence and using the fact that  $(z_k)$ is weakly null, we can assume that $y^*_n$ is $w^*$-converging, and after subtracting
            its $w^*$ limit and possibly replacing $\delta_0$ by a smaller number
            we can assume that $(y^*_n)$ is $w^*$ null.

           After passing again to subsequences, we
           can assume that there exist $p_n$'s and $q_n$'s with
           $$\|P^{\bFb^*}_{(p_n,q_n)}(y_n^*)-y^*_n\|\le \vp_n$$
           and $q_n+n< p_{n+1}$ for $n\in\N$.
           Then we let $z^*_n=P^{\bFb^*}_{(p_n,q_n)}(y_n^*)/(1+\vp)$,  and deduce that
            $\|z^*_n\|\le 1$ and
           $z^*_n(z_n)\ge \delta_0/(1+\vp))=:\delta_0'$.

           In both cases we found  $z^*_n\in \oplus_{p_n+1}^{q_n-1} F^*_j$, with
           $\|z_n^*\|\le 1$, $z^*_n(z_n)\ge \delta_0'$
           and $z^*_n|_{\psi(X)}=0$ in the first case and
           $\|z^*_n|_{\psi(X)}\|\le \vp_n$ in the second.

           By Proposition \ref{P:1.4b} we find
           $b^*_n\in \ell_1(\Gammab_{q_n-1}\setminus \Gammab_{p_n})$, for $n\in\N$
           so that $\|b^*_n\|_{\ell_1}\le \Mb$ and  $z^*_n=P_{(p_n,q_n)}^{\bFb^*}(b^*_n)$.

           Using now the density assumption of $B_{(p_n,q_n)}$  we can choose
           $\tilde b_n^*\in B_{(p_,q_n)}$ with $\|\tilde b_n^*-\frac{1}{\Mb}b^*_n\|\le \vp_{q_n}/(2M+4)\le\vp_{q_n}/2\Mb$, since $\Mb\le M\vee2$.
           So if we let $\tilde z^*_n=P_{(p_n,q_n)}^{\bFb^*}(\tilde b^*_n)$, we deduce
           that $\|z^*_n/\Mb-\tilde z^*_n\|\leq 2\Mb\vp_{q_n}/2\Mb=\vp_{q_n}$ and hence $\tilde z^*_n(z_n)\ge z^*_n(z_n)/\Mb -\|z^*_n/\Mb-\tilde z^*_n\|\ge \delta_0'/\Mb- \vp_n$, for all $n\in\N$.

            Let $n_0\in\N$ be such that  $\delta_0'\ge  2\vp_{n_0}\Mb$.
            It is enough to show that $(z_n)_{n\ge n_0}$
            has lower $(v_{q_n})_{n\ge n_0}$ estimates.
           We can therefore assume without loss of generality that $n_0=1$.
            Let  $(\alpha_j)_{j=1}^N\subset \R$ with $\|\sum_{j=1}^N \alpha_j v_{q_j}\|=1$
            and using Lemma \ref{L:4.1} (in the unconditional case)  we can choose
            $(\beta_ j)_{j=1}^N\subset \R$ with $\sum_{j=1}^N \beta_j v^*_{q_j}\in D^V$
             so that
             $$\sum_{j=1}^N \beta_j v^*_{q_j}\Big(\sum_{j=1}^N \alpha_j v_{q_j})= \sum_{j=1}^N \alpha_j\beta_j\ge ( 1-\vp).$$

           Since $(p_n)$ and $(q_n)$ satisfy the assumptions of Lemma \ref{L:6.3}, we can choose
             $\gammab\in\Lambda$ so that
           $$ e^*_\gammab \Big(\sum_{j=1}^N \alpha_j  z_j\Big)=
            \sum_{j=1}^N  \alpha_j \beta_j P^{\bFb^*}_{(p_j,q_j)} (e^*_\gammab)(z_j)=
                     c\sum_{j=1}^N  \alpha_j   \beta_j z^*_j(z_j)\ge c(1-\vp)\delta_0'/2\Mb,
      $$
           which finishes the proof of (b) and (c) and thus Theorem
           \ref{T:6.1} in full.
                  \end{proof}
 We now prove Theorem B.
    \begin{proof}[Proof of Theorem B]
  Let $X$ and $U$ be totally incomparable  spaces with separable duals.

     By Theorem \ref{thm:FOSZ} $U$ embeds into a space $W$ with an FDD which  satisfies subsequential $T_{c,\alpha}$-upper
       estimates for some $\alpha<\omega_1$ and some $0<c<1$.
      As noted before we can assume that, after possibly replacing $\alpha$ by one of its  powers, we can assume that $c\le \sfrac1{16}$.
We also noted that Proposition 7 in \cite{OSZ2} calculates the
Szlenk index of $T_{\alpha,c}$ to be
$Sz(T_{\alpha,c})=\omega^{\alpha\omega}$. We may thus choose
$\beta>\alpha$ so that $Sz(T_{\beta,c})>Sz(T_{\alpha,c})$.
Furthermore, any infinite dimensional subspace of $T_{\alpha,c}$ has
the same Szlenk index as $T_{\alpha,c}$.  We immediately have that
$T_{\alpha,c}$ and $T_{\beta,c}$ are totally incomparable, that is
no infinite dimensional subspace of $T_{\alpha,c}$ is isomorphic to
a subspace of $T_{\beta,c}$.  This idea can be refined further to
give that no normalized block sequence in $T_{\alpha,c}$ dominates a
normalized block sequence in $T_{\beta,c}$.

  Using Theorem A and    Remark \ref{R:5.3} we can embed $X$ into a Bourgain-Delbaen space $Y$ with shrinking FDD
  $\bF=(F_j)$ so that $X\cap\coo(\oplus_{j=1}^\infty F_j)$ is dense in $X$.
We apply now Theorem \ref{T:6.1} to $Y$,  with  $(v_j)$ being the
unit vector basis of $T_{c,\beta}$, to obtain a Bourgain-Delbaen
space $Z$, and an embedding $\psi$ of $X$ into $Z$, so that every
normalized  block sequence,  which has a positive distance to
$\psi(X)$, has a subsequence  $(z_i)$ which dominates some
subsequence of $(v_j)$.  If $(z_i)$ is equivalent to a basic
sequence in $U$, then $(z_i)$ is dominated by a subsequence of the
unit vector basis for  $T_{c,\alpha}$.  Thus a subsequence of the
unit vector basis for $T_{\alpha,c}$ must dominate a subsequence of
$(v_i)$ (the unit vector basis for $T_{\beta,c}$), which is a
contradiction.  Thus no normalized block sequence in $Z$,  which has
a positive distance to $\psi(X)$, is equivalent to a subsequence in
$U$.

Now  any normalized sequence in $Z$ has a subsequence which is
equivalent to a sequence in $X$ or has a subsequence which has a
positive distance to $\psi(X)$.  In both cases it follows that
 the sequence is not equivalent to a sequence in $U$.  Theorem B
 follows.
  \end{proof}

\begin{proof}[Proof of Theorem C]
Assume that $X$ is reflexive. Using Theorem \ref{thm:OSZ2} we can assume that $X$ has an FDD $(E_i)$ which satisfies
 for some $\alpha<\omega_1$ both  subsequential   $T_{\alpha,c}$-upper and subsequential $T^*_{\alpha,c}$-lower
  estimates. As noted before we can assume that $c\le \sfrac1{16}$.

  By Theorem \ref{T:4.9}  we can embed $X$ into a Bourgain-Delbaen space $Y$ with a shrinking FDD $\bF=(F_j)$,
  associated to a sequence of Bourgain-Delbaen sets $(\Delta_n)$,
  via the mapping $\psi$ given in (\ref{E:5.2}).

  Now we apply Theorem \ref{T:6.1} (b) to the unit vector basis $(v_j)$  of $T^*_{\alpha,c}$ and obtain  an augmentation $(\Theta_n)$ of $(\Delta_n)$ generating
   a Bourgain-Delbaen space $Z$ having an FDD $\bFb\!=\!(\Fb_j)$, so that every  normalized block basis $(z_n)$ in $Z$
    has a subsequence which is either equivalent to a block sequence in $X$, or which dominates   a subsequence of
    $(v_j)$.  Moreover, the later case holds for all normalized
    block bases of $(z_n)$.
    In both cases it follows that this subsequence is boundedly complete, and since it is  shrinking it follows that
    it must span a reflexive space.
\end{proof}

Similarly we can show the following result, whose proof we ommit.
\begin{thm}\label{T:6.4} Let $X$ be a Banach space with separable dual and let $(u_j)$ be a shrinking basic sequence,
none of whose subsequences is equivalent to a sequence in $X$. Then
$X$ embeds into a Bourgain-Delbaen space $Z$ whose dual is
isomorphic to $\ell_1$, and which does not contain any sequence
which is equivalent to any subsequence of $(u_j)$.
\end{thm}
Using a construction similar to one in the proof of Theorem \ref{T:6.1} we can show the following
 embedding result for spaces with an FDD satisfying subsequential lower estimates.
\begin{thm}\label{T:6.5}
Let $V$ be a Banach space with a  normalized unconditional
 basis  $(v_i)$,  having the following  property.
 \begin{align}\label{E:6.4.1} &\text{There is a constant $C>0$ so that for any two sequences $(p_n)$ and $(q_n)$ in $\N$,}\\
  &\text{with $p_1<q_1<p_2<q_2<\ldots $, $(v_{p_n})$ $C$-dominates $(v_{q_n})$.}\notag
 \end{align}
 Let $X$ be a Banach space with an
FDD $(E_i)$ which satisfies  subsequential $V$-lower
estimates. Then $X$ embeds into a $\cL_\infty$ space $Z$ with an FDD
$(\Fb_i)$ which satisfies skipped subsequential $V'$-lower estimates
where $V'$ is some subsequence of $V$. Furthermore, if $(E_i)$ and
$(v_i)$ are both shrinking, then $(\Fb_i)$ can be chosen to be
shrinking too.
\end{thm}

\begin{proof}
After renorming, we may assume that the FDD $\bE=(E_i)$ is bimonotone
and that the basis $(v_i)$ is 1-unconditional. We use the
construction of Section \ref{S:4} to define a $\cL_\infty$ space $Y$
with an FDD $\bF=(F_i)$ and an embedding $\phi:X\rightarrow Y$ such that
$\phi(E_i)\subset F_{m_i}$ for some sequence
$(m_i)\in[\N]^{\omega}$. For convenience, we will refer to the space
$\phi(X)$ as $X$. As the FDD $(E_i)$ satisfies  subsequential
$V$-lower estimates, there exists
$K\geq1$, so that
\begin{align}\label{E:6.4.2}
&\text{if $(x_i)\subset X$ is a normalized  block sequence
such that $x_i\in\oplus_{j=m_{p_i}}^{m_{q_{i}}}F_{j}$,}\\
&\text{ with $1=p_1<q_1<p_2,\ldots$, then $(x_i)$
$K$-dominates $(v_{q_i})$.}\notag
\end{align}
  We now define the Banach space
$\tilde{V}\cong V\oplus c_0$ with basis $(\tilde{v}_i)$ given by
 $\tilde{v}_{m_i}=v_{i}$  and $\tilde v_i=e_i$ if $i\not\in\{m_j\}$,
where $(e_i)$ is the unit vector basis of $c_0$.  It is clear that
$(\tilde{v}_i)$ is a 1-unconditional normalized basic sequence, and
that $(\tilde{v}_i)$ is shrinking if $(v_i)$ is
shrinking.

 We denote
the projection constant of $(F_i)$ by $M$. The sets $(\Deltab_n)$,
$\Theta^{(0,1)}$, $\Theta^{(0,2)}$, $\Theta^{(1,1)}$, and
$\Theta^{(1,2)}$ are defined as in Theorem \ref{T:6.1} for
some constant $c<1/{K}$, the  basic sequence $(\tilde v_i)$, and some inductively
chosen   $\vp_{n+1}/(2M+4)$-dense sets
$B_{(k,n]}\subset B_{\ell_1(\Gammab_n\setminus\Gammab_k)}$
(i.e. we are using the case "no assumptions on $X$").  This
construction  yields that $(\Deltab_n)$ admits an associated
Bourgain-Delbaen space $Z$ with FDD $\bFb=(\Fb_j)$
       whose decomposition\ constant $\Mb$ is  not larger than  $\max(M, 1/(1-2c))\le \max (M,2)$.
       If $(F_j)$ and
       $(v_n)$ are both shrinking in $V$, and thus, the optimal $c$-decompositions of elements of $B_{\tilde{V}^*}$  are admissible with respect to some
        compact subset of $[\N]^{<\omega}$, we have that the FDD $\bFb=(\bFb)$  is shrinking in
        $Z$.  Furthermore, we have an isometric embedding $\psi: X\rightarrow Z$.

Before continuing, we need the following lemma which is analogous to
Lemma \ref{L:6.3}.
\renewcommand{\qedsymbol}{}
             \end{proof}

         \begin{lem}\label{L:6.5}
       Let $(z^*_j)$ be a   block basis in $Z^*$ with respect to
       $\bFb^*$ such that there exist integers $p_1<q_1<p_2<q_2...$
       with
$\supp_{\bFb^*}(z^*_n)\subset(m_{p_{n}},m_{q_{n}})$ for all
$n\in\N$. Assume that
$$
 z^*_n=P^{\bFb*}_{(m_{p_n},m_{q_n})}(\tilde z_n^*)\text{
for some $\tilde z_n^*\in B_{(m_{p_{n}},m_{q_n})}$,}, \text{ for $n\in\N$}.
$$
           Then for any  sequence  $(\beta_j)_{j=1}^N$ with $w^*=
\sum_{j=1}^{N}\beta_j v^*_{{q_j}} \in D^V$, there exists
        $\gammab\in \Lambda_{N+ k_N }$  so that
       \begin{equation}  \label{E:6.5.0a} P^{\bFb^*}_{(m_{p_n},m_{q_n})}(e^*_\gammab)\!=\!c \beta_n z^*_n,  \text{ if $n\!\le\!N$, and } P^{\bFb^*}\!\!(e^*_\gammab)(\psi(x))\!=\!\sum_{n=1}^N c \beta_n z^*_n(\psi(x))
       \text{ if $x\kin X$}  .
        \end{equation}
              \end{lem}
Since  parts of the proof  are essentially the same as the proof of Lemma \ref{L:6.3} we will only sketch it and point out where
both proofs differ.
       \begin{proof}[Sketch]We will prove our claim by induction on $N$ and the case $N=1$ is exactly like in the proof of Lemma \ref{L:6.3}
       (with $p_j$ and $q_j$ being replaced by $m_{p_j}$ and $m_{q_j}$, respectively). To show the claim
        for $N+1$, assuming the claim to be true for $N$, we let
         $w^*=\sum_{j=1}^{N+1} \beta_j \tilde v_{m_{q_j}}\!=\!
         \sum_{j=1}^{N+1} \beta_j  v_{{q_j}}\kin D^{\tilde V}$, and
         define $\ell\kin\N$, $\ell\ge 2$ and $\gammab_j$ and $\etab_j$, $j\!=\!1,2\ldots,\ell$, as in Lemma \ref{L:6.3}. We need only to show
         by induction on $j=1,2\ldots \ell$, that  $|e^*_{\gammab_j}(\psi(x))|\le \|x\|$ for $x\in X$ (without the assumption of
         Lemma \ref{L:6.3} that
         $|z^*_j(\psi(x))|\le\delta_j\|x\|$, for $j\le \ell$).
          Using the induction hypothesis on the $\etab_j$'s, we deduce by
induction on   $j=1,\ldots \ell$ that for $x\in X$
\begin{align*}
|e^*_{\gammab_j}(\psi(x))|=&|c^*_{\gammab_j}(\psi(x))|\\
 &\leq \sum_{n=1}^{N_j}
|c\beta_n z^*_n(\psi(x))|\\
 &\leq \sum_{n=1}^{N_j}
c|\beta_n| \big\|P^{\bFb}_{(m_{p_{n}},m_{q_{n}})}(\psi(x))\big\|\\
 &= c\left(\sum_{n=1}^{N_j}
\beta_n v^*_{q_n}\right)\left(\sum_{n=1}^{N_j}\|P^{\bFb}_{(m_{p_{n}},m_{q_{n}})}(\psi(x))\|v_{q_n}\right)\\
&\leq c\Big\|\sum_{n=1}^{N_j}\|P^{\bFb}_{(m_{p_{n}},m_{q_{n}})}(\psi(x))\|\tilde v_{m_{q_n}}\Big\|\\
&\leq  c\Big\|\sum_{n=1}^{N_j}\big(\|P^{\bFb}_{(m_{p_{n}},m_{q_{n}})}(\psi(x))\|\tilde v_{m_{q_n}} + \|P^{\bFb}_{[m_{q_{n}},m_{p_{n+1}}]}(\psi(x))\| \tilde v_{m_{p_{n+1}}}  \big)   \Big\|\\
&\le cK\|x\|\le\|x\|
\end{align*}
(in the penultimate line we use  the 1-unconditionality of $(\tilde v_j)$ and in the case of $j=\ell$ we
put $p_{_{N_\ell+1}}=m_{q_{_{N_{\ell}+1}}}$,
for the last line  we use \eqref{E:6.4.2})
and thus $\gammab_1\in \Theta_{m_{q_1}}^{(0,1)}$, if
$|\supp(w^*_1)|=1$, and $\gammab_1\in
\Theta_{m_{p_{_{N_1+1}}}}^{(0,2)}$, if $|\supp(w^*_1)|>1$, and
  $\gammab_j\in \Theta_{m_{q_{_{N_j}}}}^{(1,1)}$, if  $|\supp(w^*_1)|=1$,  and $\gammab_j\in \Theta_{m_{q_{_{N_j}}}+N_j-N_{j-1}+1}^{(1,2)}$, if  $|\supp(w^*_1)|>1$, if $j=2,3\ldots \ell$.
  We put then $\gammab=\gammab_\ell$, and the rest of the proof follows again like in Lemma \ref{L:6.3}        .
       \end{proof}

       \begin{proof}[Continuation of  the Proof of Theorem \ref{T:6.4}] To finish the proof
           we consider a normalized block basis  $(z_n)$ in $Z$ such that there exists sequences $p_1<q_1<p_2<q_2\ldots$
       with
$\supp_{\bFb}(z_n)\subset(m_{p_{n}},m_{q_{n}})$ for all $n\in\N$.
 We choose $z^*_n\in\oplus_{j\in(p_n,q_n)} \Fb^*_j$, with
   $\|z^*_n\|\le1$, $      z^*_n(z_n)\ge \frac{1}{2\Mb}$.

           By Proposition \ref{P:1.4b} there exists
           $b^*_n\in \ell_1(\Gammab_{q_n-1}\setminus \Gammab_{p_n})$, for $n\in\N$
           so that $\|b^*_n\|_{\ell_1}\le \Mb$ and  $z^*_n=P_{(p_n,q_n)}^{\bFb^*}(b^*_n)$.
           Using the density assumption of $B_{(p_n,q_n)}$,  we choose
           $\tilde b_n^*\in B_{(p_,q_n)}$ with $\|\tilde b_n^*-\frac{1}{\Mb}b^*_n\|\le \vp_{q_n}/(2M+4)\le\vp_{q_n}/2\Mb$, since $\Mb\le M\vee2$.
           So if we let $\tilde z^*_n=P_{(p_n,q_n)}^{\bFb^*}(\tilde b^*_n)$, we deduce
           that $\|z^*_n/\Mb-\tilde z^*_n\|\leq 2\Mb\vp_{q_n}/2\Mb=\vp_{q_n}$ and hence $\tilde z^*_n(z_n)\ge z^*_n(z_n)/\Mb -\|z^*_n/\Mb-\tilde z^*_n\|\ge 1/\Mb- \vp_n$, for all $n\in\N$.

            Let  $(\alpha_j)_{j=1}^N\subset \R$ with $\|\sum_{j=1}^N \alpha_j v_{q_j}\|=1$
            and using Lemma \ref{L:4.1} (in the unconditional case)  we can choose
            $(\beta_ j)_{j=1}^N\subset \R$ with $\sum_{j=1}^N \beta_j v^*_{q_j}\in D^V$
             so that
             $$\sum_{j=1}^N \beta_j v^*_{q_j}\Big(\sum_{j=1}^N \alpha_j v_{q_j})= \sum_{j=1}^N \alpha_j\beta_j\ge ( 1-\vp).$$

           Since $(p_n)$ and $(q_n)$ satisfy the assumptions of Lemma \ref{L:6.3}
           (recall that $m_{j+1}=j+m_j$), we can choose
             $\gammab\in\Lambda$ so that
           $$ e^*_\gammab \Big(\sum_{j=1}^N \alpha_j  z_j\Big)=
            \sum_{j=1}^N  \alpha_j \beta_j P^{\bFb^*}_{(p_j,q_j)} (e^*_\gammab)(z_j)=
                     c\sum_{j=1}^N  \alpha_j   \beta_j z^*_j(z_j)\ge c(1-\vp)(1/\Mb-\vp),
      $$
           which gives that $(z_n)$
dominates $(v_{q_n})$.  Thus we may block the FDD $(\Fb_i)$ to
achieve the theorem.
 \end{proof}

\end{document}